\documentclass{article}
\usepackage[utf8]{inputenc}
\usepackage{amsmath}
\usepackage{amsfonts}
\usepackage{mathtools}
\usepackage{bm}
\usepackage[a4paper, margin=1in]{geometry}
\usepackage{amssymb}
\usepackage{amsthm}
\usepackage{stmaryrd}
\usepackage{mdframed}
\usepackage{enumitem}
\usepackage{dsfont}
\usepackage{xcolor}
\usepackage[hidelinks]{hyperref}

\usepackage{natbib}
\setcitestyle{authoryear, open={((},close={)}}

\mathtoolsset{showonlyrefs}

\newtheorem{theorem}{Theorem}
\newtheorem{definition}{Definition}
\newtheorem{lemma}{Lemma}
\newtheorem{result}{Result}
\newenvironment{example}[1]{%https://www.overleaf.com/project/60e48663254e662cf12f87ef
    \begin{mdframed}[topline=false,bottomline=false,rightline=false,innertopmargin=0pt,linewidth=2pt,linecolor=lightgray]
    \paragraph{#1.}}{%
    \end{mdframed}}

\newcommand{\comment}[1]{}

\newcommand{\bepsilon}{\bm{\epsilon}}

\newcommand{\bbeta}{\bm{\beta}}
\newcommand{\bzero}{\bm{0}}

\newcommand{\gradient}[1]{\nabla_{\bm{#1}}}

\newcommand{\evalat}[1]{\bigg \rvert_{#1}}

\newcommand{\PsiP}{\Psi_P}
\newcommand{\Psizero}{\psi_0}

\newcommand{\model}{\mathcal{M}}
\newcommand{\reals}{\mathbb{R}}
\newcommand{\E}{\mathbb{E}}
\newcommand{\Peps}{P_{\epsilon}}

\DeclareMathOperator*{\argmin}{\arg\,\min}

\DeclarePairedDelimiter\norm{\lVert}{\rVert}

\newcommand{\MCOsymb}{\sharp}

\newcommand{\CATEsymb}{\sharp}
\newcommand{\QbarFluctuation}[2]{\bar{Q}_{{#1},{#2}}}
\newcommand{\QFluctuation}[2]{Q_{{#1},{#2}}}
\newcommand{\cleverH}{H}

\newcommand{\tmleLoss}{\mathcal{L}}

\newcommand{\Beps}{B^\circ_P}

\newcommand{\tmleBeta}{\hat{\bbeta}_n^*}

\newcommand{\evalepszero}{\bigg \rvert_{\epsilon = 0, \bbeta = \Beps(0)}}

\newcommand{\ci}{\mathrel{\text{\scalebox{1.07}{$\perp\mkern-10mu\perp$}}}}

\title{Inference in Marginal Structural Models by Automatic Targeted Bayesian and Minimum Loss-Based Estimation}

\author{Herbert Susmann$^1$, Antoine Chambaz$^2$\\[1em]
  $^1$ Department of Biostatistics \& Epidemiology, University of Massachusetts Amherst\\
  $^2$ Université Paris Cité, CNRS, MAP5, F-75006 Paris, France
}
\date{\today}

\begin{document}

\maketitle

\section*{Abstract}
Two of the principle tasks of causal inference are to define and estimate the effect of a treatment on an outcome of interest. Formally, such treatment effects are defined as a possibly functional summary of the data generating distribution, and are referred to as \textit{target parameters}. Estimation of the target parameter can be difficult, especially when it is high-dimensional. Marginal Structural Models (MSMs) provide a way to summarize such target parameters in terms of a lower dimensional working model. We introduce the semi-parametric efficiency bound for estimating MSM parameters in a general setting. We then present a frequentist estimator that achieves this bound based on Targeted Minimum Loss-Based Estimation. Our results are derived in a general context, and can be easily adapted to specific data structures and target parameters. We then describe a novel targeted Bayesian estimator and provide a Bernstein von-Mises type result analyzing its asymptotic behavior. We propose a universal algorithm that uses automatic differentiation to put the estimator into practice for arbitrary choice of working model. The frequentist and Bayesian estimators have been implemented in the Julia software package \texttt{TargetedMSM.jl}. Finally, we illustrate our proposed methods by investigating the effect of interventions on family planning behavior using data from a randomized field experiment conducted in Malawi. 

\section{Introduction}
Causal inference is concerned with defining and estimating the effect of a treatment on an outcome of interest. For example, consider the \textit{Conditional Average Treatment Effect} parameter. Let $O_{1:n}$ be $n$ i.i.d. draws from a distribution $P_0$ of a generic variable $O = (X, A, Y)$ where $X$ is a vector of covariates, $A$ is a binary treatment indicator, and $Y$ a binary outcome. The CATE is defined as the following functional summary of $P$,
\begin{align}
    \Psi_P^{\CATEsymb}(X) := \E_P\left[Y \middle| A = 1, X\right] - \E_P\left[Y \middle| A = 0, X\right],
\end{align}
and is interpretable as the expected difference in outcome given treatment vs. non-treatment within strata of covariates $X$.

Estimating $\Psi^{\CATEsymb}_P(X)$ may be challenging, especially when $X$ is high-dimensional. One way to proceed would consist in assuming (recklessly) that $\Psi^{\CATEsymb}_P(X)$ belongs to a known low-dimensional parametric model. Suppose there are $V \subset X$ potential \textit{treatment effect modifiers}. A parametric model for the marginal distribution of $\Psi_P^{\CATEsymb}(X)$ that assumes a linear functional form conditional on $V$ could be
\begin{align}
    \label{eq:example-msm}
    \E_P[\Psi_P^{\CATEsymb}(X) \mid V ] = \text{MSM}_{\bbeta}(V) := \bbeta^\top (1, V)^\top.
\end{align}
The target of estimation would then be the parameter $\bbeta \in \mathcal{B}$, which is of lower dimension than $\Psi_P^\CATEsymb(X)$. This approach was originally introduced by \cite{robins1998marginal}; in this line of research, the assumed parametric model on the marginal distribution of the target parameter is referred to as a \textit{Marginal Structural Model} (MSM).

In practice, the correct functional form of the parameter of interest will not be known. Any parametric model that is adopted will usually be misspecified. Non-parametric Marginal Structural Models build on MSMs without assuming correct specification \citep{neugebauer2007nonparametric, van2011targeted}. The idea is to summarize the parameter of interest in terms of a lower-dimensional working model with respect to a loss function. For example, in the spirit of \eqref{eq:example-msm}, we could define a new parameter $B(P)$ as
\begin{align}
    B(P) := \argmin_{\bbeta \in \mathcal{B}}  \E_P \left[ \left( \Psi_P^{\CATEsymb}(X) - \text{MSM}_{\bbeta}(V) \right)^2 \right].
\end{align}
That is, $\text{MSM}_{B(P)}(V)$ is the best linear approximation of $\Psi_P^{\CATEsymb}(X)$ in terms of $V$ with respect to squared-error loss function. Note that, in contrast to the first approach, the definition of $B(P)$ introduces no statistical assumptions on the form of $\Psi_P^{\CATEsymb}$. We adopt this non-parametric approach to MSMs in the current work, providing a general definition and notation in Section \ref{section:msm-definition}.

Continuing in the same spirit, we seek to estimate $B(P_0)$ as non-parametrically as possible, without making any parametric assumptions on the distribution $P_0$. We only assume that $P_0$ is within a non-parametric model $\mathcal{M}$ of laws defined on the support of the observed data. The goal is to construct an asymptotically normal and efficient estimator of $B(P_0)$. The first step is to derive the \textit{efficient influence function} (EIF) $D^*(P)$ of the target parameter at any $P \in \mathcal{M}$. Knowledge of the EIF is key because its variance defines the semi-parametric efficiency bound of estimating $B(P_0)$, and it serves as a building block for constructing efficient  non-parametric estimators. We review relevant ideas from semi-parametric efficiency theory and present the form of the EIF for $B$ in Section \ref{section:efficiency-theory}. 

Next, we draw on the Targeted Learning framework, specifically Targeted Minimum Loss-Based Estimation (TMLE), to build an efficient estimator of $B(P_0)$ \citep{van2011targeted, van2018targeted}. Suppose we have an initial estimate of $P_0$, $P_n^\circ\in\model$. While the plug-in estimator $B(P_n^\circ)$ will typically be biased, the core insight of TMLE is that it is possible to mitigate this bias and achieve asymptotic normality by carefully fluctuating the initial estimate $P_n^\circ$. In broad strokes, we define a parametric fluctuation $P^\circ_{n,\bepsilon}$ with parameter $\bepsilon \in \mathbb{R}^p$. An estimate $\bepsilon_n^*$ of $\bepsilon$ is found by minimizing an empirical risk associated with a carefully chosen loss function. An updated estimate of $B(P_0)$ is then derived by plugging in the fluctuated law, as $B(P^\circ_{n,\bepsilon_n^*})$. The key property of the fluctuation and loss function, from which the desirable properties of TMLE are derived, is that the fluctuated law empirically solves the efficient influence function $D^*$; that is,
\begin{align}
    \frac{1}{n}\sum_{i=1}^n D^*(P^\circ_{n,\bepsilon_n^*})(O_i) = o_p(n^{-1/2}).
\end{align}
It is for this reason that the estimator is called \textit{targeted}, because it solves the efficient influence function of the parameter of interest.
Under conditions on the estimator $P_n^\circ$, it is possible to show that the updated estimate is asymptotically normal and efficient with asymptotic variance given by the variance of the efficient influence function. 

Estimators based on TMLE have been developed for a number of marginal structural models in various contexts. These include estimators for vector-valued treatment assignments, time-varying treatments, and instrumental variable designs \citep{stein1956efficient,rosenblum2010targeted,petersen_tmle2014, toth2016tmle,zheng2016doubly,zheng2018marginal}. In practice, users may wish to test several alternative working model specifications and loss functions. This is difficult to do using existing approaches, as the estimators (and related software) are constructed for specific classes of working models and loss functions. Our results build off these works by proposing a TMLE based estimation strategy in a generalized data structure and for arbitrary MSM loss functions and working models. In Section \ref{section:tmle} we provide a blueprint for building targeted estimators for MSMs with arbitrary working models and loss functions. Our quite general software implementation is made possible by an algorithm that adapts to the choice of loss function and working model using automatic differentiation \citep{baydin2018automatic}. 

Statistical inference can be loosely categorized into frequentist and Bayesian approaches, corresponding to different interpretations of probability \citep{bayarri2004interplay, hajek2019probability}. TMLE was developed within the frequentist paradigm, drawing on a long line of research in  semi-parametric efficiency theory; see \cite{ bickel1993efficient,vanDerVaart1996weak,vandervaart2000asymptotic} among many others. A Bayesian version of TMLE was first discussed in \cite{van2008construction}, based on the observation that, if a likelihood is adopted for the parametric fluctuation $P^\circ_{n,\bepsilon}$, then the parameter $\bepsilon$ can be estimated via Bayesian inference. Such an approach was later extended to estimation of average treatment effects and class proportions in an unlabeled dataset \citep{diaz2011targeted, diaz2020nonparametric}. We build on this research, and describe a novel Bayesian estimator in Section \ref{section:bayesian-tmle} that is the first Bayesian targeted estimator to be proposed for MSMs.

%As a motivating example, we consider the problem of estimating the association between exposure to advertisements for family planning on adoption of modern contraceptive methods in a cohort of women in Kenya, using data collected by the PMA-Agile program \citep{tsui2020pmaagile} and analyzed previously by \citep{anglewicz2021pma}. We define a Marginal Structural Model to summarize how the treatment effect of advertisement on adoption of modern contraceptive methods depends on participant age. The results are presented using both frequentist and Bayesian approaches in Section \ref{section:motivating-application}.
As an illustrative example, we consider the problem of estimating the association between a multi-faceted family planning intervention on contraceptive use in a cohort of women in Malawi using data collected in a randomized field experiment \citep{karra2022fptrial}. We define a Marginal Structural Model to summarize how the treatment effect depends on the participant's number of children. The results are presented using both frequentist and Bayesian approaches in Section \ref{section:motivating-application}.

In summary, the rest of the article proceeds as follows. In Section \ref{section:msm-definition} we provide a general definition and notation for Marginal Structural Models. In Section \ref{section:efficiency-theory} we use results from semi-parametric efficiency theory to derive the efficiency bound for estimating MSMs in non-parametric models. In Section \ref{section:tmle} we present an estimator based on Targeted Minimum Loss-Based Estimation that achieves the efficiency bound. In Section \ref{section:bayesian-tmle} we develop a novel targeted Bayesian estimator for MSMs. In Section \ref{section:computation} we discuss details of the software implementation. In Section \ref{section:simulation} we share the results of a simulation study. In Section \ref{section:motivating-application} apply the estimators to an example. Finally, we conclude with a discussion in Section \ref{section:discussion}. 

\section{Marginal Structural Models}
\label{section:msm-definition}

\subsection{Statistical Viewpoint}
\label{sec:stat:view}
Recall that we have $n$ i.i.d. draws $O_{1:n}$ from a distribution $P_0$ of a generic variable $O$. Let $\mathcal{O}$ be the support of $P_0$. We assume only that $P_0$ falls in a non-parametric model $\mathcal{M}$ of laws on $\mathcal{O}$. Further, suppose $O$ can be decomposed into two variables, $O = (X, Z)$ (this is always possible, as $Z$ can be set to the trivial set). Let $\mathcal{X}$ and $\mathcal{Z}$ denote the spaces to which $X$ and $Z$ belong. For any $P \in \mathcal{M}$, let $\PsiP: \mathcal{X} \to \mathbb{R}$ be a functional summary of $P$, with $\Psi_P$ a member of the function class $\mathcal{F}$. For later convenience, we introduce $\mathcal{T} := \cup_{P \in \model} \PsiP(\mathcal{X})$ (that is, the union of the images of all $\PsiP$s). We focus on instances where $\mathcal{T}$ is an open set. We also assume that $\Psi_P$ depends on $P$ through a finite tuple of parameters $\bar{Q}_P = \left( \bar{Q}^{(0)}_P, \dots, \bar{Q}^{(J)}_P \right)$.

\begin{example}{Example}
Conditional Average Treatment Effect. Let $Z = (A, Y)$, where $A$ is a binary treatment indicator and $Y$ is a binary outcome. Let $g_P(a, x) := P\left(A = a \middle| X = x\right)$, and suppose that $g_P(a, X) > 0$ holds $P$-almost-surely for all $P \in \mathcal{M}$ and $a \in \{0,1\}$. Let $\bar{Q}_P^{(a)}(X) := \E_P\left[ Y \middle| A = a, X\right]$, which is defined almost surely on $\mathcal{X}$. Define $\PsiP^\CATEsymb(x) := \bar{Q}^{(1)}_P(x) - \bar{Q}^{(0)}_P(x)$, interpretable as the expected difference in the outcome for treatment vs. non-treatment within the strata of covariates $X = x$. Appendix \ref{section:supplementary-example} considers the case when $Y \in \reals$.
\end{example}

\paragraph{} For convenience, we will write $\Psizero := \Psi_{P_0}$ to denote the functional under the true data generating distribution $P_0$. Estimating the functional $\Psizero$ across its entire domain may be challenging, especially if $\mathcal{X}$ is high dimensional. The idea with MSMs is to seek an approximation of the functional $\Psizero$ defined by a lower-dimensional working model. In other words, instead of estimating $\Psizero$, we estimate a parameter for the working model that leads to the best possible approximation of $\Psizero$ within the working model as measured by the risk induced by a well-chosen, problem-specific loss function.

Formally, a loss function is a function $L : \reals \times \reals \to \mathbb{R}$ such that, for all $f \in \mathcal{F}$, 
\begin{align}
    \mathbb{E}_P\left[ L(\Psi_P(X), \Psi_P(X)) \right] \leq  \mathbb{E}_P\left[ L(\Psi_P(X), f(X)) \right].
\end{align}
 The working model is a collection $\{ m_{\bbeta} : \bbeta \in \mathcal{B} \}$ of functions $m_{\bbeta} : \mathcal{X} \to \mathbb{R}$ with $\mathcal{B}$ a parameter of dimension $p$. For convenience, we write $L_m(t, \bbeta)(X) := L(t(X), m_{\bbeta}(X))$ for all $t \in \mathcal{T}$ and $\bbeta \in \mathcal{B}$.
 The target parameter $B(P) \in \mathcal{B}$ is defined as any solution to the following optimization problem:
\begin{align}
    \label{eq:optimization-problem}
    B(P) \in \argmin_{\bbeta \in \mathcal{B}} \mathbb{E}_{P} \left[ L_m(\PsiP, \bbeta)(X) \right].
\end{align}
Note that $B(P)$ depends on $P$ only through the functional summary $\PsiP$ and the marginal distribution of $X$, which we denote $Q$.

From now on, we assume that there is a unique minimizer of \eqref{eq:optimization-problem} in the case where $P = P_0$. We call this minimizer $\bbeta_0 := B(P_0)$. Moreover, we focus on loss functions $L$ such that $\bbeta \mapsto L_m(\Psizero, \bbeta)(X)$ is twice differentiable with a Hessian at $\bbeta_0$ denoted by $\ddot{L}_m(\Psizero, \bbeta)(X)$. In a nutshell, it therefore holds that $\bbeta_0$ solves the implicit equation
\begin{align}
    \label{eq:Up-definition}
    \bm{0} = \frac{\partial}{\partial \bbeta} \mathbb{E}_{P_0}\left[  L_m(\Psizero, \bbeta)(X) \right] \evalat{\bbeta = \bbeta_0}.
\end{align}
Assuming that the gradient and expected value can be interchanged, then $\bbeta_0$ also solves the equation
\begin{align}
    \label{eq:implicit-eq-zero}
    \bm{0} = \mathbb{E}_{P_0}\left[\ddot{L}_m(\Psizero, \bbeta_0)(X) \right].
\end{align}
This equality will play an important role in the semi-parametric analysis of the target functional $B$.

\begin{example}{Example (cont'd)}
Let $V \subset X$ be a set of ($p - 1$) \textit{treatment effect modifiers}. 
We adopt a linear working model with an intercept term
\begin{align}
    \label{eq:linear-working-model}
    m_{\bbeta} : X \mapsto \bbeta^\top (1, V)^\top \quad \text{ for all } \bbeta \in \mathcal{B} := \reals^p,
\end{align}
and a squared-error loss function given by $L(a, b) := (a-b)^2$, hence
\begin{align}
    \label{eq:squared-error-loss}
    L_m(\PsiP^\CATEsymb, \bbeta)(X) = (\PsiP^\CATEsymb(X) - m_{\bbeta}(X))^2.
\end{align}
The target parameter is then given by
\begin{align}
    B^\CATEsymb(P) = \argmin_{\bbeta \in \mathcal{B}} \E_P\left[ \left( \PsiP^\CATEsymb(X) - \bbeta^\top (1, V)^\top\right)^2 \right].
\end{align}
Note that when $V = \varnothing$ the parameter $B^\CATEsymb(P)$ reduces to the Average Treatment Effect (ATE):
\begin{align}
    B^\CATEsymb(P) &= \E_P\left[ \PsiP^\CATEsymb(X) \right] = \E_P\left[\bar{Q}_P^{(1)}(X) - \bar{Q}_P^{(0)}(X) \right].
\end{align}
\end{example}

\subsection{Causal Viewpoint}
While the main focus of this work is statistical, we review briefly the causal interpretation of the target parameter $B(P)$. 
The parameter $B(P)$ can be viewed as a projection of the parameter $\Psi_P$ onto a lower-dimensional space. As such, $B(P)$ is causally identifiable under the same conditions for which $\Psi_P$ is identifiable. Formally, let $\mathbb{P}$ the causal law from which $P$ arises, that is,  $\mathbb{P}$ is the law of the complete data (including unobservable potential outcomes) and $P$ is a joint marginal law thereof.  Let $\Psi_\mathbb{P}^C$ be a causal functional summary of $\mathbb{P}$ and $B^C(\mathbb{P})$ be the causal analog to $B(P)$, defined as
\begin{align}
    \label{eq:psi-causal}
    B^C(\mathbb{P}) = \argmin_{\bbeta \in \mathcal{B}} \E_P\left[ L_m(\Psi_\mathbb{P}^C, \bbeta)(X) \right].
\end{align}
Suppose that there are a set of identification assumptions sufficient to show that $\Psi_P(X) = \Psi^C_\mathbb{P}(X)$. It is then straightforward to see that $B(P) = B^C(\mathbb{P})$ under the same conditions by direct substitution of $\Psi_P(X)$ for $\Psi_\mathbb{P}^C(X)$ in \eqref{eq:psi-causal}.

\begin{example}{Example (cont'd)}
    The parameter $\Psi^{\CATEsymb}_P$ is identifiable under standard causal assumptions which are reviewed below for completeness. Let $Y(a)$ be the potential outcome under treatment assignment $a \in \{ 0, 1 \}$. Define the causal parameters
    \begin{align}
        \Psi^{C,\CATEsymb}_\mathbb{P}(X) &:= \E_{\mathbb{P}}[Y(1) \mid X] - \E[Y(0) \mid X], \\
        B^{C, \CATEsymb}(\mathbb{P}) &:= \argmin_{\bbeta \in \mathcal{B}} \E_\mathbb{P}\left[ \left( \Psi_{\mathbb{P}}^{C, \CATEsymb}(X) - \bbeta^\top (1, V)^\top\right)^2 \right].
    \end{align}
    We seek conditions under which $B^{C, \CATEsymb}(\mathbb{P}) = B^{\CATEsymb}(P)$. Assume
    \begin{enumerate}
        \item \label{assumption-consistency} Consistency: $Y = Y(A)$. 
        \item  \label{assumption-positivity} Positivity: $g_P(a, X) > 0$ for both $a\in\{0,1\}$, $P$-almost surely.
        \item \label{assumption-no-unmeasured-confounders} No unmeasured confounders: under $\mathbb{P}$, $Y(a) \ci A \mid X$ for both $a \in \{ 0, 1 \}$.
    \end{enumerate}
    Under these conditions,
    \begin{align}
        \Psi_{\mathbb{P}}^{C, \CATEsymb}(X) &= \E_{\mathbb{P}}[Y(1) \mid X] - \E_{\mathbb{P}}[Y(0) \mid X] \\
        &= \E_{\mathbb{P}}[Y(1) \mid A = 1, X] - \E_{\mathbb{P}}[Y(0) \mid A = 0, X] \text{ (assumptions \ref{assumption-positivity} and \ref{assumption-no-unmeasured-confounders})} \\
        &= \E_P[Y \mid A = 1, X] - \E_P[Y \mid A = 0, X] \text{ (assumption \ref{assumption-consistency})} \\
        &= \Psi^{\CATEsymb}_P(X).
    \end{align}
    It then follows directly that $B^{C, \CATEsymb}(\mathbb{P}) = B^{\CATEsymb}(P)$.
\end{example}

\section{Semi-Parametric Efficiency Theory}
\label{section:efficiency-theory}

\paragraph{} Our goal is to find an asymptotically efficient estimator of $\bbeta_0 := B(P_0) \in \reals^p$, the value of the parameter of interest under the true data generating distribution $P_0$. We first review some semi-parametric efficiency theory through which we can derive the non-parametric efficiency bound for estimating $\bbeta_0$ \citep{vdv02}.

In this section, we derive the semi-parametric efficiency bound for estimating $\bbeta_0$ in a non-parametric model $\model$ (that is, $\model$ is the set of all laws $P$ on $\mathcal{X} \times \mathcal{Z}$ such that $B(P)$ is well-defined). The semi-parametric efficiency bound for estimating $B(P)$ is defined via the concept of the hardest parametric submodel. For any $P\in\model$, for any integer $k \geq 1$ and for all 
\begin{align}
    s \in \mathcal{S} := \left\{  h \in (L_0^2(P))^k : h \neq 0, \norm{h}_\infty < \infty, \E_P[h(O) h(O)^\top] \text{ invertible} \right\},
\end{align}
define a parametric submodel $ \mathcal{P}_s := \{ P_{s,\epsilon} : \bepsilon \in \reals^k, \norm{\bepsilon}_\infty < \infty\} \subset \model$, characterized by $dP_{\epsilon, s} = (1 + \bepsilon^\top s) dP$. The invertibility condition in the definition of $\mathcal{S}$ amounts to a condition that $\mathcal{P}_s$ is identifiable. Note that $P_{s, \epsilon} = P$ at $\epsilon = 0$, and the score of   $P$ equals $s$ at $\epsilon = 0$. That is, $\mathcal{P}_s$ is a fluctuation of $P$ in the direction $s$.

Assume the derivative $\frac{\partial}{\partial\bepsilon} B(P_{\epsilon, s})$ exists at $\bepsilon = \bzero$ for any $s \in \mathcal{S}$. 
The Cramér-Rao bound for estimating $B(P)$ within the submodel $\mathcal{P}_{s}$ (that is, the lowest asymptotic variance possible for an unbiased estimator of $B(P)$, in the sense of the Loewner order: for any $A$, $B$ Hermitian matrices, $A \geq B$ if $A - B$ is positive semi-definite) is given by
\begin{align}
    \left[ \frac{\partial}{\partial \bepsilon} B(P_{\epsilon, s}) \evalat{\bepsilon = \bzero} \right]^\top \E_{P}[s(O) s(O)^\top]^{-1} \left[ \frac{\partial}{\partial \bepsilon} B(P_{\epsilon, s}) \evalat{\bepsilon = \bzero} \right].
\end{align}
Assume further that there exists a continuous, bounded, linear map $\dot{B}$ mapping $\mathrm{closure}(\mathcal{S})$ to $\reals^{k \times p}$ such that, for all $s \in \mathcal{S}$, 
\begin{align}
    \frac{\partial}{\partial \bepsilon} B(P_{\epsilon, s}) \evalat{\bepsilon = \bzero} &= \dot{B}(P_{\epsilon, s}).
\end{align}
We then say that $B$ is \textit{pathwise differentiable} at $P$ with respect to $\{ \mathcal{P}_s : s \in \mathcal{S} \}$.
By the Riesz representation theorem, the derivative can be expressed as
\begin{align}
    \frac{\partial}{\partial\bepsilon} B(P_{\epsilon, s}) \evalat{\bepsilon = \bzero} = \E_P[s(O) D^*(P)(O)^\top]
\end{align}
for a function $D^*(P) \in (L_0^2(P))^p$ which is called the \textit{efficient influence function} (EIF) of the parameter $B$ at $P$. For later convenience, let $\lambda^*(P) := D^*(P) D^*(P)^\top$. We can then rewrite the Cramér-Rao bound as
\begin{align}
    \E_P[s(O) D^*(P)(O)^\top]^\top \E_P[s(O) s(O)^\top]^{-1} \E_P[s(O) D^*(P)(O)^\top].
\end{align}
The \textit{hardest parametric submodel}, i.e. the submodel for which estimating $B(P)$ is the hardest, is the submodel with the largest Cramér-Rao bound (for the Loewner order). Slightly tedious algebra and the Cauchy-Schwarz inequality imply that the largest  Cramér-Rao bound is
\begin{align}
    \E_P[\lambda^*(P)(O)].
\end{align}
The semi-parametric efficiency bound for estimating $B(P)$ within a non-parametric model $\model$ is then defined as $\E_P[\lambda^*(P)(O)]$. Deriving the form of $D^*(P)$ is therefore crucial if we wish to construct efficient non-parametric estimators of $B(P)$. 

In the rest of the section the parametric submodel $\mathcal{P}_s$ is a theoretical tool used to derive the form of $D^*$, so for simplicity we can choose $\bepsilon$ to be of dimension one. In the next theorem, we present detailed conditions under which $B(P)$ is pathwise differentiable, and characterize the form of $D^*(P)$. 

\begin{result}[Efficient Influence Function of $B$]
\label{theorem:general-eif:short}
    Let us make the following assumption: for all $P \in \model$ and for all $s \in L^2_0(P)$ such that $s \neq 0$, $\norm{s}_\infty < \infty$, if $\left\{ P_\epsilon : |\epsilon| < \norm{s}_\infty^{-1} \right\} \subset \model$ is characterized by $\frac{dP_\epsilon}{dP} = 1 + \epsilon s$ for every $\epsilon \in \reals$ such that $|\epsilon| < \norm{s}_\infty^{-1}$, 
    then the mapping $\epsilon \mapsto \Psi_{P_\epsilon}(X)$ is almost surely differentiable at 0 and
    \begin{align}
        \frac{d}{d\epsilon} \Psi_{\Peps}(X) \evalat{\epsilon = 0} = \E_P \left[\Delta^*(P)(O) s(O) \mid X \right]
    \end{align}
    for a function $\Delta^*(P) \in L_0^2(P)$. Under further assumptions stated in Theorem~\ref{theorem:general-eif:full}, the target functional $P \mapsto B(P)$ is pathwise differentiable at every $P \in \model$, with an efficient influence function $D^*(P)$ given by
\begin{align}
    D^*(P)(O) &= M^{-1} \left[
        D_1^*(P)(O) + D_2^*(P)(X)
    \right],
\end{align}
where $D_1^*(P), D_2^*(P) \in L_0^2(P)$ are given by
\begin{align}
    D_1^*(P)(O) &= \gradient{} \dot{L}(\Psi_P(X), B(P))(X) \times \Delta^*(P)(O), \\
    D_2^*(P)(X) &= \dot{L}(\Psi_P(X), B(P))(X),
\end{align}
and the normalizing matrix $M$ is given by
\begin{align}
    M = -\mathbb{E}_{P}\left[\ddot{L}_m(\PsiP(X), B(P))(X) \right].
\end{align}
\end{result}
The full statement (Theorem~\ref{theorem:general-eif:full}) and its proof are given in Appendix \ref{proof:general-eif}.

\begin{example}{Example (cont'd)}
We can derive the efficient influence function $D^*_\CATEsymb(P)$ for $B^\CATEsymb$ at $P$ by applying Theorem \ref{theorem:eif-example}. First, we prove a lemma giving the conditional efficient influence function $\Delta_{\CATEsymb}^*(P)$ for~$\Psi_P^\CATEsymb$ at $P$.

\begin{lemma}[Conditional Efficient Influence Function of $\Psi^{\CATEsymb}$]
\label{lemma:eif-delta-example}
The functional summary $\Psi^{\CATEsymb}$ satisfies the first assumption of Result~\ref{theorem:general-eif:short} with a conditional efficient influence function $\Delta^*_{\CATEsymb}(P) \in L_0^2(P)$  at any $P\in\model$  given by
\begin{align}
    \Delta^*_{\CATEsymb}(P)(O) = \left\{ \frac{\mathbb{I}(A = 1)}{g_P(1, X)} - \frac{\mathbb{I}(A = 0)}{g_P(0, X)}  \right\} (Y - \bar{Q}^{(A)}_P(X)).
\end{align}
\end{lemma}
The proof is given in Appendix \ref{proof:eif-delta-example}. Note that the form of $\Delta^*_{\CATEsymb}$ is recognizable as a part of the efficient influence function for the Average Treatment Effect \citep[Chapter 5]{van2011targeted}. With this result in hand, we are ready to state the efficient influence function for the parameter $B^\CATEsymb$.

\begin{theorem}[Efficient Influence Function of $B^{\CATEsymb}$]
\label{theorem:eif-example}
Suppose that Assumptions \eqref{assumption:implicit-function-theorem}-\eqref{assumption:bounded-derivative} of Theorem~\ref{theorem:general-eif:full} are satisfied. Then the target functional $P \mapsto B^{\CATEsymb}(P)$ is pathwise differentiable at every $P \in \model$, with an efficient influence function $D_{\CATEsymb}^*(P)$ given by
 \begin{align}
    D^*_{\CATEsymb}(P)(X, A, Y) = M^{-1} \left\{ D^*_{1,\CATEsymb}(P)(X, A, Y) + D^*_{2,\CATEsymb}(P)(X) \right\},
\end{align}
where 
\begin{align}
    D^*_{1,\CATEsymb}(P)(X, A, Y) &= \left\{ \frac{I(A = 1)}{g_P(1, X)} - \frac{I(A = 0)}{g_P(0, X)} \right\} (Y - \bar{Q}^{(A)}_P(X))(1, V)^\top, \\
    D^*_{2,\CATEsymb}(P)(X) &= (\Psi^{\CATEsymb}_P(X) - B^\CATEsymb(P)^\top (1, V)^\top)(1, V)^\top,
\end{align}
and the normalizing matrix $M$ is given by
\begin{align}
    M = -\mathbb{E}_{P}\left[(1, V)^\top (1, V)\right].
\end{align}
\end{theorem}
The proof is provided in Appendix \ref{proof:eif-example}. 

\begin{paragraph}{Remark} The form of the EIF resembles closely that of the Average Treatment Effect (ATE). In fact, when $V = \varnothing$ then $D_\CATEsymb^*(P)$ reduces to the EIF of the ATE:
\begin{align}
    D^*_{1,\CATEsymb}(P)(X, A, Y) &= \left\{ \frac{I(A = 1)}{g_P(1, X)} - \frac{I(A = 0)}{g_P(0, X)} \right\} (Y - \bar{Q}^{(A)}_P(X) ), \\
    D^*_{2,\CATEsymb}(P)(X) &= \Psi^{\CATEsymb}_P(X) - \E_P\left[ \Psi^{\CATEsymb}_P(X) \right].
\end{align}
\end{paragraph}

\begin{paragraph}{Remark}
The variance of the EIF is given by 
\begin{align}
    \E_P[\lambda^*(P)(O)] &= M^{-1} \E_P\left[ \left( \frac{\mathrm{Var}(Y \mid A, X)}{g_P(A, X)^2} + \left(\Psi_P^\CATEsymb(X) - B^\CATEsymb(P)^\top (1, V)^\top \right)^2 \right) (1, V)(1,V)^\top\right] M^{-1}.
\end{align}
\end{paragraph}

\end{example}

\section{Targeted Minimum Loss-Based Estimation}
\label{section:tmle}

In this section, we describe an efficient  and asymptotically normal estimator for $\bbeta_0$ based on Targeted Minimum Loss-Based Estimation (TMLE;  \cite{van2011targeted, van2018targeted}).

\paragraph{Notation} The target parameter $B(P)$ and the efficient influence function $D^*(P)$ of $B$ at $P$ only depend on $P$ through several nuisance parameters. As we will be only estimating these nuisance parameters, rather than all of $P$, it is helpful to introduce notation that emphasizes the parts of $P$ necessary for estimating the target parameter and efficient influence function.
\begin{itemize}
    \item Nuisance parameters for $B(P)$: note that $B(P)$ depends on $P$ only through (a) the marginal distribution of $X$ under $P$, which we call $Q_P$, and (b) the functional $\Psi_P$, which itself only depends on $\bar{Q}_P$ and possibly $Q_P$. For convenience, we will write $B(P)$ and $B(\bar{Q}_P, Q_{P})$ interchangeably. 
    \item Nuisance parameters for $D^*(P)$: recall that the EIF $D^*(P)$ is the sum of $D_1^*(P)$ and $D_2^*(P)$. The first component $D_1^*(P)$ may depend on $P$ through $\Psi_P$, $Q_{P}$, and additional nuisance parameters we call $\eta_P$.  The function $D_2^*(P)$ only depends on $P$ through $\Psi_P$ and $Q_{P}$. As such, we will write $D_1^*(\bar{Q}_P, Q_{P}, \eta_P)$, $D_2^*(\bar{Q}_P, Q_{P})$, and $D^*(\bar{Q}_P, Q_{P}, \eta_P)$ interchangeably with $D_1^*(P)$, $D_2^*(P)$, and $D^*(P)$. 
\end{itemize} 
For convenience, let  $\mathcal{Q} = \{ Q_{P} : P \in \mathcal{M} \}$ denote the parameter space of $Q_{P}$, denote $Q_{0} := Q_{P_0}$, and define $\eta_0 := \eta_{P_0}$.

Suppose we have initial estimates $\bar{Q}_n^0 \in \mathcal{Q}$, $Q_{n}^0$, and $\eta_n$ of $\bar{Q}_0$, $Q_{0}$, and $\eta_0$. From now on, we will always choose $Q_{n}^0$ to be the empirical distribution of $X_1, \dots, X_n$, but we do not yet choose a particular estimating strategy for $\bar{Q}_n^0$ and $\eta_n$. Given these initial estimates, we could estimate $\bbeta_0$ via the plug-in estimator $B(\bar{Q}_n^0, Q_{n}^0)$ (which is indeed well-defined). However, although $\bar{Q}_n^0$ and $Q_{n}^0$ may be good estimators of $\bar{Q}_0$ and $Q_{0}$, there is no guarantee that plugging them into the target parameter yields a good estimator of $\bbeta_0$. The idea of TMLE is to reduce the bias of the plug-in estimator by iteratively updating the initial estimates $\bar{Q}_n^0$ and $Q_{n}^0$ so as to target the parameter of interest $\bbeta_0$. Here, the iterative updates may incorporate the additional nuisance parameters~$\eta_n$. Given mild conditions on $\bar{Q}_n^0$ and $\eta_n$ the resulting targeted estimator is asympotically normal and efficient.

Now we describe in more detail the updating procedure at the heart of TMLE. The iterative updates are based on fluctuations of the current estimators of $\bar{Q}_0$ and $Q_{0}$. First, choose loss functions $\tmleLoss_j : \reals \times \mathcal{O} \to \reals$ corresponding to each component  $\bar{Q}^{(j)}_P$ of $\bar{Q}_P = \left( \bar{Q}^{(0)}_P, \dots, \bar{Q}^{(J)}_P \right)$ such that the following conditions are satisfied for all $j = 0, \dots, J$:
\begin{enumerate}[label={(L\arabic*)}]
    \item \label{condition: well-defined-loss} $\bar{Q}^{(j)}_0 = \argmin \left\{ \E_{P_0}[\tmleLoss_j(\bar{Q}_j(X), O)] : \bar{Q}_j \in \{ \bar{Q}^{(j)}_{P} : P \in \mathcal{M} \}\right\}$
    \item \label{condition:loss-differentiable} For each $P \in \model$, it holds $P$-almost surely that $t' \mapsto \tmleLoss_j(t', O)$ is differentiable at every $t \in \reals$ with derivative $\dot{\tmleLoss}_j(t, O) \in \reals$.
\end{enumerate}
Next, introduce a parametric fluctuation model satisfying the following conditions:
\begin{enumerate}[label={(M\arabic*)}]
    \item \label{condition:well-defined-fluctuation} For any $P \in \model$ with corresponding $\bar{Q}_P$, $Q_{P}$, and $\eta_P$, we can define a fluctuation model
    \begin{align}
        \left\{ ( \QbarFluctuation{P}{\bepsilon}^{(0)}, \dots, \QbarFluctuation{P}{\bepsilon}^{(J)}, \QFluctuation{P}{\bepsilon} )  : \bepsilon \in \mathbb{R}^p \right\}
    \end{align}
    such that:
    \begin{enumerate}
        \item If $\bepsilon = \bzero$, then $\QbarFluctuation{P}{\bepsilon}^{(j)} = \bar{Q}_P^{(j)}$ for all $j = 0, \dots, J$ and $\QFluctuation{P}{\bepsilon} = Q_{P}$.
        \item It holds $P$-almost surely that the mappings $\bepsilon \mapsto \tmleLoss_j(\QbarFluctuation{P}{\bepsilon}^{(j)}(X), O)$ ($j = 0, \dots, J)$ and $\bepsilon \mapsto -\log Q_{P,\bepsilon}(X)$ are differentiable at $\bepsilon = \bzero$, and $D^*(P)(O)$ belongs to
        \begin{align}
            \label{eq:condition-submodel-gradient}
            \mathrm{Span} \left( 
                \frac{\partial}{\partial \bepsilon} \tmleLoss_0(\QbarFluctuation{P}{\bepsilon}^{(0)}(X), O) \evalat{\bepsilon = \bzero}, 
                \dots, \frac{\partial}{\partial \bepsilon} \tmleLoss_J(\QbarFluctuation{P}{\bepsilon}^{(J)}(X), O) \evalat{\bepsilon = \bzero}, 
                \frac{\partial}{\partial \bepsilon} \left(-\log \QFluctuation{P}{\bepsilon}\right)(X) \evalat{\bepsilon = \bzero} 
            \right).
        \end{align}
    \end{enumerate}
\end{enumerate}
Given loss functions and a fluctuation model satisfying the above conditions, conduct the following iterative procedure:
\begin{enumerate}
    \item Start with the initial estimates $\bar{Q}_n^{(j), 0}$ ($j = 0, \dots, J$), $Q_{n}^0$, $\eta_n$. 
    \item For $k \geq 1$, recursively let $\bar{Q}_n^{(j), k} = \QbarFluctuation{n}{\bepsilon_n^k}^{(j), k-1}$  ($j = 0, \dots, J$) and $Q_{n}^k = \QFluctuation{n}{\bepsilon_n^k}^{k-1}$, where 
    \begin{align}
        \bepsilon_n^k = \argmin_{\bepsilon \in \reals^p} \frac{1}{n}\sum_{i=1}^n \left[ \sum_{j = 0}^J \left[ \tmleLoss_j(\QbarFluctuation{n}{\bepsilon}^{(j), k-1}(X_i), O_i) \right] - \log \QFluctuation{n}{\bepsilon}^{k-1}(X_i) \right],
    \end{align}
    and stop when $\norm{\bepsilon^k_n} \approx 0$ (rigorously, when $\bepsilon^k_n = o_P(n^{-1/2})$).
    \item Set $\bar{Q}_n^* = \left( \bar{Q}_n^{(0), k}, \dots, \bar{Q}_n^{(J), k} \right)$ and $Q_{n}^* = Q_{n}^k$, where $k$ is the final iteration of the above step. 
\end{enumerate}
The targeted estimator is the plug-in estimator given by $\tmleBeta = B(\bar{Q}_n^*, Q_{n}^*)$. The desirable statistical properties enjoyed by this estimator derive from the fact that the final update approximately solves the estimating equation $P_n\left[ D^*(\bar{Q}_n^*, Q_{n}^*, \eta_n) \right] = o_P(n^{-1/2})$.\footnote{From now on we will write $P[f]$ for the integral $\int f dP$.} The following theorem states conditions under which $\bbeta_n^*$ is asymptotically normal and efficient, which serves as the basis for conducting valid statistical inference.
\begin{theorem}[Asymptotic normality and efficiency of $\bbeta_n^*$]
\label{theorem:asymptotic-normality}
Let us assume that 
\begin{enumerate}
    \item Rate of convergence of second-order remainder: $\tmleBeta - \bbeta_0 + P_0 D^*(P_n^*) = o_P(n^{-1/2})$.
    \item Donsker conditions: $D^*(P_n^*)$ is in a $P_0$-Donsker class with probability tending to one, and the random squared norm $P_0[(D^*(P_n^*) - D^*(P_0))^2] = o_P(n^{-1/2})$.
\end{enumerate}
Then the estimator $\tmleBeta$ is asymptotically linear with the form \begin{align}
  \sqrt{n}\left(\tmleBeta - \bbeta_0 \right) = \frac{1}{\sqrt{n}} \sum_{i=1}^n D^*(P_0)(O_i) + o_P(1),
\end{align}
which implies that $\tmleBeta$ is asymptotically normal and efficient: \begin{align}
    \sqrt{n}\left(\tmleBeta - \bbeta_0 \right) \rightsquigarrow N\left(0, P_0[\lambda^*(P_0)] \right).
\end{align}
\end{theorem}
The proof is given in Appendix \ref{proof:tmle-asymptotic-normality}. Note that the Donsker conditions can be relaxed through the use of cross-validation \citep{zheng2010asymptotic}.

So far we have not given an explicit form for the loss functions $\tmleLoss_0, \dots, \tmleLoss_{J}$  nor the fluctuation model. The explicit forms of the $\tmleLoss_0, \dots, \tmleLoss_J$ will depend on the structure of $O$ and the form of the functional summary $\PsiP$. Likewise, we cannot give a complete specification of the fluctuation model as it too depends on the choices for $\tmleLoss_0, \dots, \tmleLoss_J$ and the form of $\PsiP$. In addition, there may be multiple fluctuation models that satisfy the required conditions. We give a blueprint for one particular fluctuation model satisfying the required conditions that can be easily adapted to specific problems. This blueprint also yields a useful constraint on the form of the $\tmleLoss_0, \dots, \tmleLoss_J$. 

For any $P \in \model$ with corresponding $\PsiP$, $\bar{Q}_P$, $Q_{P}$, and $\eta_P$, characterize the fluctuation model by setting
\begin{align}
    \label{eq:blueprint-1}
    \QbarFluctuation{P}{\bepsilon}^{(j)}(O) &= \phi^{-1}\left(\phi(\bar{Q}_P^{(j)}(O)) + \cleverH_j(O)\bepsilon^\top \gradient{} \dot{L}_m(\PsiP, B(P))(X)\right) \quad \text{for each } j = 0, \dots, J, \\
    \label{eq:blueprint-2}
    \text{and } \quad \QFluctuation{P}{\bepsilon}(X) &= C(\bepsilon) \exp\left( \bepsilon^\top \dot{L}_m(\PsiP, B(P))(X) \right) Q_P(X),
\end{align}
where the so-called ``clever-covariates", which need to be chosen, are functions $\cleverH_j : \mathcal{O} \to \reals$ ($j =0, \dots, J$) and the constant $C(\bepsilon)$ is chosen such that $Q_{P, \bepsilon}$ is well-defined. We also assume that the transformation $\mathrm{link}$ has an inverse $\phi^{-1}$, and that this inverse has derivative $\dot{\phi}^{-1}$. It is easy (sic) to see that $\QbarFluctuation{P}{\bzero}^{(j)} = \bar{Q}_P^{(j)}$ ($j = 0, \dots, J$) and $\QFluctuation{P}{\bzero} = Q_P$, so (M1a) is satisfied. Next we derive a constraint under which condition (M1b) is satisfied. First, see that 
\begin{align}
    \frac{\partial}{\partial \epsilon} \log \QFluctuation{P}{\bepsilon}(X) \evalat{\bepsilon = \bzero} &= \dot{L}_m(\PsiP, B(P))(X) \\
    &= D_2^*(\PsiP, Q_P)(X).
\end{align}
In addition, by condition \ref{condition: well-defined-loss} we can compute
\begin{align}
    \frac{\partial}{\partial \bepsilon} \tmleLoss_j(\QbarFluctuation{P}{\bepsilon}^{(j)}(O), O) \evalat{\bepsilon = \bzero} &=  \dot{\tmleLoss}_j(\QbarFluctuation{P}{\bepsilon}^{(j)}(O), O) \evalat{\bepsilon = \bzero} \frac{\partial}{\partial \bepsilon} \QbarFluctuation{P}{\bepsilon}^{(j)}(O) \evalat{\bepsilon = \bzero} \\
    &= \dot{\tmleLoss}_j(\bar{Q}_P^{(j)}(O), O) \times \dot{g}^{-1}(g(\bar{Q}_P^{(j)}(O)) \times H_j(O) \times \gradient{} \dot{L}_m(\PsiP, B(P))(X) .
\end{align}
Recall that condition (M1b) requires that $D^*(P)(O)$  be included in the linear span of the two gradients from the two above displays. Therefore, the only way for (M2b) to be satisfied is if 
\begin{align}
    \sum_{j = 0}^J \dot{\tmleLoss}_j(\bar{Q}_P^{(j)}(O), O) \dot{\phi}^{-1}(\phi(\bar{Q}_P^{(j)}(O))) H_j(O) \propto \Delta^*(P)(O).
\end{align} In practice, $\Delta^*$ often includes a term resembling a residual. The loss functions $\tmleLoss_j$ ($j = 0, \dots, J$) are then chosen such that their summed derivatives equals the residual term, and the remaining part of $\Delta^*$ is integrated into the $H_j$ ($j = 0, \dots, J$).

\begin{example}{Example (cont'd)}
Now that we are working in the context of a particular example we can fill in the details of the targeted estimator. First, note that $\Psi^{\CATEsymb}_P$ depends on $P$ only through the parameters $\bar{Q}_P = (\bar{Q}_P^{(0)}, \bar{Q}_P^{(1)})$, and $D_{1,\CATEsymb}^*(P)$ depends on $P$ through the additional nuisance parameters $\eta_P(X) = (g_P(0, X), g_P(1, X))$. 
For any $P \in \model$ with corresponding $\PsiP$, $\bar{Q}_P$, $Q_{P}$ and $\eta_P$, characterize the fluctuation model as follows, using the transformation $\phi=\mathrm{logit}$:
\begin{align}
    \mathrm{logit}\left(\QbarFluctuation{P}{\bepsilon}^{(0)}(X)\right) &= \mathrm{logit}\left(\bar{Q}_P^{(0)}(X)\right) + \cleverH_0(O) \bepsilon^\top (1, V)^\top , \\
    \mathrm{logit}\left(\QbarFluctuation{P}{\bepsilon}^{(1)}(X)\right) &= \mathrm{logit}\left(\bar{Q}_P^{(1)}(X)\right) + \cleverH_1(O) \bepsilon^\top (1, V)^\top , \\
    \QFluctuation{P}{\bepsilon}(X) &= C(\bepsilon) \exp\left( \bepsilon^\top  (\psi^{\CATEsymb}_P(X) - B(P)^\top (1, V)^\top) (1, V)^\top \right) Q_X(X).
\end{align}
The clever covariates are given by
\begin{align}
    \cleverH_0(O) &= -\frac{I(A = 0)}{g_P(0, X)}, \\
    \cleverH_1(O) &= \phantom{-}\frac{I(A = 1)}{g_P(1, X)}.
\end{align}
Note that the fluctuations are quite similar to those commonly adopted to define a TMLE for the ATE \citep[Chapter 5]{van2011targeted}. The loss functions are chosen to be characterized by
\begin{align}
    \tmleLoss_0(t, O)   &= -I(A = 0)\left[ Y \log\left(t\right) + (1 - Y)\log\left(1 - t\right) \right], \\
    \tmleLoss_1(t, O)   &= -I(A = 1)\left[ Y \log\left(t\right) + (1 - Y)\log(1 - t) \right]
\end{align}
(both $t \in (0, 1)$).
To ensure this is a valid setup, we need to check conditions (L1), (L2), and (M1). \ref{condition: well-defined-loss} is satisfied for this choice of log-likelihood loss function \citep{van2011targeted}. The derivatives of the loss functions are 
\begin{align}
    \dot{\tmleLoss}_0(t, O) &= -I(A = 0)\left[ \frac{Y - t}{(1 - t)t} \right], \\
    \dot{\tmleLoss}_1(t, O) &= -I(A = 1)\left[ \frac{Y - t}{(1 - t)t} \right]
\end{align}
(both $t \in (0, 1)$).
Therefore, \ref{condition:loss-differentiable} is satisfied. To check \ref{condition:well-defined-fluctuation}, see that, for $\phi = \mathrm{logit}$, $\dot{\phi}^{-1}(\phi(t)) = (1 - t) t$.
Therefore
\begin{align}
    &\dot{\tmleLoss}_0(\bar{Q}_P^{(0)}(X), O) \dot{\phi}^{-1}(\phi(\bar{Q}_P^{(0)}(X))) \cleverH_0(O) +  \dot{\tmleLoss}_1(\bar{Q}_P^{(1)}(X), O) \dot{\phi}^{-1}(\phi(\bar{Q}_P^{(1)}(X))) \cleverH_1(O) \\
    &= -\left\{ \frac{I(A = 1)}{g_P(1, X)} - \frac{I(A = 0)}{g_P(0, X)} \right\} (Y - \bar{Q}_P^{(A)}(X)) = -\Delta_P^*(X),
\end{align}
which shows that \ref{condition:well-defined-fluctuation} is satisfied. 
\end{example}

%\subsection{Cross-Validated TMLE}
%The Donsker conditions of Theorem \ref{theorem:asymptotic-normality} can be relaxed through the use of cross-validation. 

\section{Bayesian Inference}
\label{section:bayesian-tmle}
The key observation behind Bayesian Targeted Maximum Likelihood Estimation is that when the TMLE loss functions can be interpreted as log-likelihoods, then the optimization step in the iterative procedure of TMLE simply corresponds to maximum likelihood estimation of $\bepsilon$. Furthermore, due to the likelihood interpretation, we can also use Bayesian inference to estimate $\bepsilon$ via a straightforward application of Bayes' theorem. In this section, we formalize this application of Bayesian inference, and present an oracle Bernstein von-Mises type result that suggests the resulting Bayesian targeted estimator converges asymptotically to the truth with optimal variance given by the variance of the efficient influence function. By oracle result, we mean that the Berstein von-Mises theorem concerns an idealized version of the Bayesian targeted estimator that we define below.

Suppose a set of loss functions $\mathcal{L}_j$, $j = 0, \dots, J$ has been chosen such that conditions \ref{condition: well-defined-loss} and \ref{condition:loss-differentiable} are satisfied. In addition, suppose we have already chosen for any $P \in \model$ with corresponding $\bar{Q}_P$, $Q_P$, and $\eta_P$ a well-defined fluctuation submodel
 \begin{align}
    \left\{ ( \QbarFluctuation{P}{\bepsilon}^{(0)}, \dots, \QbarFluctuation{P}{\bepsilon}^{(J)}, \QFluctuation{P}{\bepsilon} )  : \bepsilon \in \mathbb{R}^p \right\}
\end{align}
satisfying condition \ref{condition:well-defined-fluctuation}.  In this section, in order to reframe Targeted Maximum Likelihood Estimation in a Bayesian framework, we focus on the case where it is possible, for any $P \in \model$, to define an additional nuisance parameter $\sigma_P$ and a submodel $ \mathcal{F} = \{ F_{\bepsilon} : \bepsilon \in \reals^p \} \subset \model $ such that the log-likelihood of $O$ under each $F_{\bepsilon}$ writes as
\begin{align}
    \label{eq-likelihood-condition}
    \log f(O \mid \bepsilon) = h_1(\sigma_P(O)) - h_2(\sigma_P(O)) \sum_{j=1}^J \mathcal{L}_j(\bar{Q}_{P,\bepsilon}^{(j)}(O), O) + \log Q_{P,\bepsilon}(X) + \text{constant},
\end{align}
where $h_1$ and $h_2$ are transformations of the nuisance parameter $\sigma_P$. It is often possible to construct such $F_{\bepsilon}$ satisfying the above condition by choosing a conditional probability distribution for $Y$ given $X$ with a conditional log-likelihood resembling the TMLE loss functions. 

\begin{example}{Example (cont'd)}
The TMLE loss functions $\mathcal{L}_0$ and $\mathcal{L}_1$ in this example can be interpreted as the negative conditional log-likelihoods of $Y$ under Bernoulli distributions. This suggests constructing $F_{\bepsilon}$ such that $Y$ given $A$ and $X$ follows a Bernoulli distribution.
Fix arbitrarily $P \in \model$ and let $F_{\bepsilon}$ be characterized by
\begin{align}
    X &\sim Q_{P, \bepsilon}, \\
    A \mid X &\sim \mathrm{Bernoulli}(1/2), \\
    Y \mid X, A &\sim \mathrm{Bernoulli}\left(\bar{Q}_{P, \bepsilon}^{(A)}(X)\right),
\end{align}
where
\begin{align}
    Q_{P,\bepsilon}(X) &= C(\bepsilon) \exp\left( M^{-1} \bepsilon^\top (\Psi^{\MCOsymb}_P(X) - B(P)^\top (1, V)^\top)(1, V)^\top \right) Q_X(X), \\
    \mathrm{logit}(\bar{Q}_{P, \bepsilon}^{(0)}(X)) &= \mathrm{logit}(\bar{Q}_P^{(0)}(X)) + H_0(O) \bepsilon^\top M^{-1} (1, V)^\top, \\
    \mathrm{logit}(\bar{Q}_{P, \bepsilon}^{(1)}(X)) &= \mathrm{logit}(\bar{Q}_P^{(1)}(X)) + H_1(O)  \bepsilon^\top M^{-1} (1, V)^\top,
\end{align}
with $M = -\E_P[(1, V)^\top (1, V) ]$.
The only difference between these fluctuations and those used for the frequentist TMLE are in the inclusion of the normalizing matrix $M$, which is included so that the conditions of the forthcoming Bernstein von-Mises theorem are satisfied. The conditional distribution of $A$ is included so that the model is fully specified, although as in the frequentist TMLE this conditional distribution is not fluctuated. No additional nuisance parameters $\sigma_P$ are needed in this example; as such, we set $h_1(\cdot) = 0$ and $h_2(\cdot) = 1$ (see Appendix \ref{section:supplementary-example} for an example where $\sigma_P$ is non-empty). The conditional log-likelihood of $O$ under $F_{\bepsilon}$ is then given by
\begin{align}
    \log f\left(O \middle| \bepsilon \right) &= Y \log\left( \bar{Q}_{P,\bepsilon}^{(A)}(X) \right) + (1 - Y)\log\left( 1 -  \bar{Q}_{P,\bepsilon}^{(A)}(X) \right) + \log Q_{P,\bepsilon}(X) + \mathrm{constant} \\
    &= \left( \mathcal{L}_0\left(\bar{Q}_{P,\bepsilon}^{(0)}(X), O\right) + \mathcal{L}_0\left(\bar{Q}_{P,\bepsilon}^{(0)}(X), O\right) \right) + \log Q_{P,\bepsilon}(X) + \mathrm{constant},
\end{align}
which satisfies \eqref{eq-likelihood-condition}.
\end{example}

Next, we discuss Bayesian estimation of the parameter $\bepsilon$ within the model $\mathcal{F}=\{F_{\bepsilon} : \bepsilon \in \mathbb{R}^{p}\}$. Let $\pi_{\epsilon}$ be a prior distribution for $\bepsilon$. Fix arbitrarily $P \in \model$, with corresponding $\bar{Q}_P$, $Q_P$, $\eta_P$, and $\sigma_P$. 
Application of Bayes' theorem yields a posterior distribution of $\bepsilon$ given by
\begin{align*}
    \Pi_{\epsilon}(\bepsilon | O_{1:n}, \bar{Q}_P, Q_P, \eta_P, \sigma_P) \propto \pi_{\epsilon}(\bepsilon)  \prod_{i=1}^n f(O_i \mid \bepsilon, \bar{Q}_P, Q_P, \eta_P, \sigma_P).
\end{align*}
We can then find a posterior distribution for $\bbeta$, the object of interest, by mapping the posterior of $\bepsilon$ onto a posterior for $\bbeta$.
For convenience, we write $\vartheta : \bepsilon \mapsto B(F_{\bepsilon})$.  The posterior distribution $\Pi_\beta$ of $\bbeta$ is then the image of $\Pi_{\epsilon}$ under $\vartheta$.

We are unlikely to have any prior information for $\bepsilon$ directly. On the contrary, we might have prior information for $\bbeta$. As such, we set a prior on $\bbeta$, which we then map back to a prior on $\bepsilon$. If $\vartheta$ has an inverse and if $\vartheta$ is differentiable with derivative $\dot{\vartheta} \in \reals^{p \times p}$, then a prior distribution $\pi_{\beta}$ for $\bbeta$ is mapped to a prior distribution $\pi_{\epsilon}$ on $\bepsilon$ by the formula
\begin{align}
    \label{eq-epsilon-prior}
    \pi_{\epsilon}(\bepsilon) = \pi_{\bbeta}\left( \vartheta(\bepsilon) \right) \left|\mathrm{det}\left(\dot{\vartheta}(\bepsilon )\right)\right|.
\end{align}

Next, we present a Bernstein von-Mises type theorem for an oracular fluctuation model that is built using the \textbf{true} nuisance parameters under $P_0$. The proof shows that in this setting, the Bayesian TMLE converges to the truth with variance given by the variance of the efficient influence function. Denote by $\mathcal{F}^0 = \{ F_{\bepsilon}^0 : \bepsilon \in \reals^p \}$ the submodel through $P_0$ satisfying \eqref{eq-likelihood-condition}. 
The following theorem only holds if the form of $F_{\bepsilon}^0$ satisfies certain conditions, one of which being that $\mathcal{F}^{0}$ be \textit{locally asymptotically normal}. We provide a definition for this property in Appendix~\ref{section:bvm-proof}, where we also state and prove Theorem~\ref{theorem:bvm:full}, the full version of the result below.

Let $N(\mu, \Sigma)$ denote the multivariate normal distribution with mean vector $\mu$ and covariance matrix~$\Sigma$. Let
\begin{align}
    \Pi_{\sqrt{n}(\beta-\beta_{0})}^{0}\left(\cdot \middle| O_{1:n} \right) &:= \Pi_{\sqrt{n}(\beta-\beta_{0})}\left(\cdot \middle| O_{1:n}, \bar{Q}_0, Q_0, \eta_0, \sigma_0 \right)
\end{align}
be the posterior for $\sqrt{n}(\bbeta - \bbeta_{0})$ corresponding to the submodel $\mathcal{F}^0$. 

\begin{result}[Oracular Bernstein von-Mises]
\label{theorem:bvm:short}
    Under assumptions stated in Theorem~\ref{theorem:bvm:full}, 
    \begin{align*}
        \norm{\Pi^0_{\sqrt{n}(\bbeta - \bbeta_0)} \left( \cdot \mid O_{1:n} \right) - N(\Delta_n^0, P_0[\lambda^*(P_0)])}_1 = o_P(1)
    \end{align*}
    where
    \begin{align}
        \Delta_{n}^0 = \frac{1}{\sqrt{n}} \sum_{i=1}^n P_0[\lambda^*(P_0)]^{-1} D^*(P_0)(O_i).
    \end{align}
\end{result}
The full statement (Theorem~\ref{theorem:bvm:full}) and its proof are given in Appendix~\ref{section:bvm-proof}.

\begin{example}{Example (cont'd)} 
  We also show in Appendix~\ref{section:bvm-proof} that the above result applies to the running example.
\end{example}

So far we have defined a family of posterior distributions of $\bbeta$ indexed by $\bar{Q}_P$, $Q_P$, $\eta_P$, and $\sigma_P$. The oracular Bernstein von-Mises proof shows that, for a fluctuation $\mathcal{F}^0$ built using the true nuisance parameters $\bar{Q}_0, Q_0, \eta_0$, and $\sigma_0$, the posterior distribution of $\bepsilon$ converges to the truth with variance given by the variance of the efficient influence function. Since in practice we have to estimate the values of these nuisance parameters, we turn to a closely related posterior from the full family of posterior distributions, the one indexed by, and conditional on, the estimates $\bar{Q}_n^*, Q_n^*$, $\eta_n$, and $\sigma_n$ --- recall that  $\bar{Q}_n^*, Q_n^*$ are the final fluctuated estimates of $\bar{Q}_0$ and $Q_0$ from the frequentist TMLE, $\eta_n$ is an estimate of the nuisance parameter $\eta$, and the newly introduced $\sigma_n$ is an estimate of $\sigma_{P_0}$. We call this posterior a \textit{targeted posterior}, and we 
hope that it inherits favorable properties from the frequentist TMLE. 

Formally, define
\begin{align}
    \label{eq-targeted-posterior2}
    \Pi^*_{\sqrt{n}(\beta-\beta_{0})}(\cdot | O_{1:n}) &:= \Pi_{\sqrt{n}(\beta-\beta_{0})}(\cdot | O_{1:n}, \bar{Q}_n^*, Q_n^*, \eta_n, \sigma_n)
\end{align}
to be the targeted
targeted posterior of $\sqrt{n}(\bbeta-\bbeta_{0})$. Under an assumption that the estimates $\bar{Q}_n^*, Q_n^*, \eta_n$, and $\sigma_n$ are consistent, it is reasonable to think that this posterior will converge to $\Pi_{\sqrt{n}(\beta-\beta_{0})}^0(\cdot \mid O_{1:n})$. By appealing to Result~\ref{theorem:bvm:short} (that is, Theorem \ref{theorem:bvm:full}), it would therefore follow that the posterior converges to a distribution centered on the MLE and with optimal asymptotic variance given by the variance of the efficient influence function. While we leave a formal proof of this convergence to future work, we conduct a simulation study in a later section to investigate the properties of $\Pi_{\beta}^*$ in a finite-sample context.

\section{Computation}
\label{section:computation}
Putting the proposed estimators into practice requires a software implementation. The frequentist and Bayesian estimators have been implemented in the software package \texttt{TargetedMSM.jl} in the Julia programming language \citep{Julia-2017}.

\subsection{Universal Algorithm}
The efficient influence function $D^*(P)$ depends on four components: the function $\Delta^*(P)$ and the derivatives $\dot{L}$, $\ddot{L}$, and $\gradient{}\dot{L}$. Each of the derivatives depends on the form of the working model and loss function chosen by the analyst. Writing these derivatives in software by hand is tedious, and makes it more difficult for users to implement new working models and loss functions.

In a unified approach, we propose a universal algorithm that encompasses all loss functions and working models that are well-chosen by the analyst. The algorithm uses automatic differentiation to compute the required derivatives, alleviating the need for the analyst to supply them 
 \citep{baydin2018automatic, margossian2019review}. When using our universal algorithm, the analyst only needs to provide software implementations of the MSM working model and loss function. Our Julia package is available at \url{https://github.com/herbps10/TargetedMSM.jl}. Its  implementation uses forward mode automatic differentiation through the Julia package \texttt{ForwardMode.jl} \citep{RevelsLubinPapamarkou2016}.

\subsection{Markov-Chain Monte Carlo}
\label{subsec:MCMC}
In practice, the targeted posterior distribution will not have a closed form. However, sampling techniques such as Markov-Chain Monte Carlo (MCMC) can be used to draw a set of samples $\bepsilon^{(1)}, \dots, \bepsilon^{(\ell)}$ from the posterior distribution of $\bepsilon$, which can then be used to generate a set of samples from the posterior distribution  of $\bbeta_0$ through the mapping $B$. In this work, we implement a Metropolis-Hastings algorithm \citep{metropolis1953equation, hastings1970monte, robert2015metropolis} with a Gaussian proposal distribution to generate samples of $\bepsilon$ which are then mapped to samples of $\bbeta$ from the posterior distribution. A detailed description of the algorithm is given in Appendix \ref{section:metropolis-hastings}. 

\subsection{Diagnostics}
\label{subsec:diagnostics}
In finite samples, it is possible that $\{ F_{\bepsilon} : \bepsilon \in \reals^p \}$ will exhibit pathological behavior in which $B(F_{\bepsilon})$ is bounded in such a way that there is no value of $\bepsilon$ such that $B(F_{\bepsilon})$ reaches the tails of the posterior distribution. However, a simple visual diagnostic can be used to identify this behavior in practice.  First, we assume that we have at hand a sample of $\ell$ draws $\bepsilon^{(1)}, \dots \bepsilon^{(\ell)}$ and $\bbeta^{(1)}, \dots, \bbeta^{(\ell)}$ from the posterior distributions of $\bepsilon$ and $\bbeta$, perhaps generated by the Metropolis-Hastings algorithm discussed in Section~\ref{subsec:MCMC}. The diagnostic simply consists in   plotting each sample with $\bbeta^{(t)}$ on the $y$-axis and
$\bepsilon^{(t)}$ on the $x$-axis. We expect to see one-to-one relationship between $\bepsilon^{(t)}$ and $\bbeta^{(t)}$. In a pathological setting, the $\bbeta^{(t)}$ reach a threshold. If this is observed, then the draws $\bbeta^{(1)}, \dots, \bbeta^{(\ell)}$ should not be trusted as an accurate characterization of the uncertainty of estimating $\bbeta_0$, as they collectively fail to capture the tails of the posterior distribution. An example of the proposed diagnostic is shown in Figure \ref{fig:tmle-diagnostic}.

\begin{figure}
    \centering
    \includegraphics[width=0.9\textwidth]{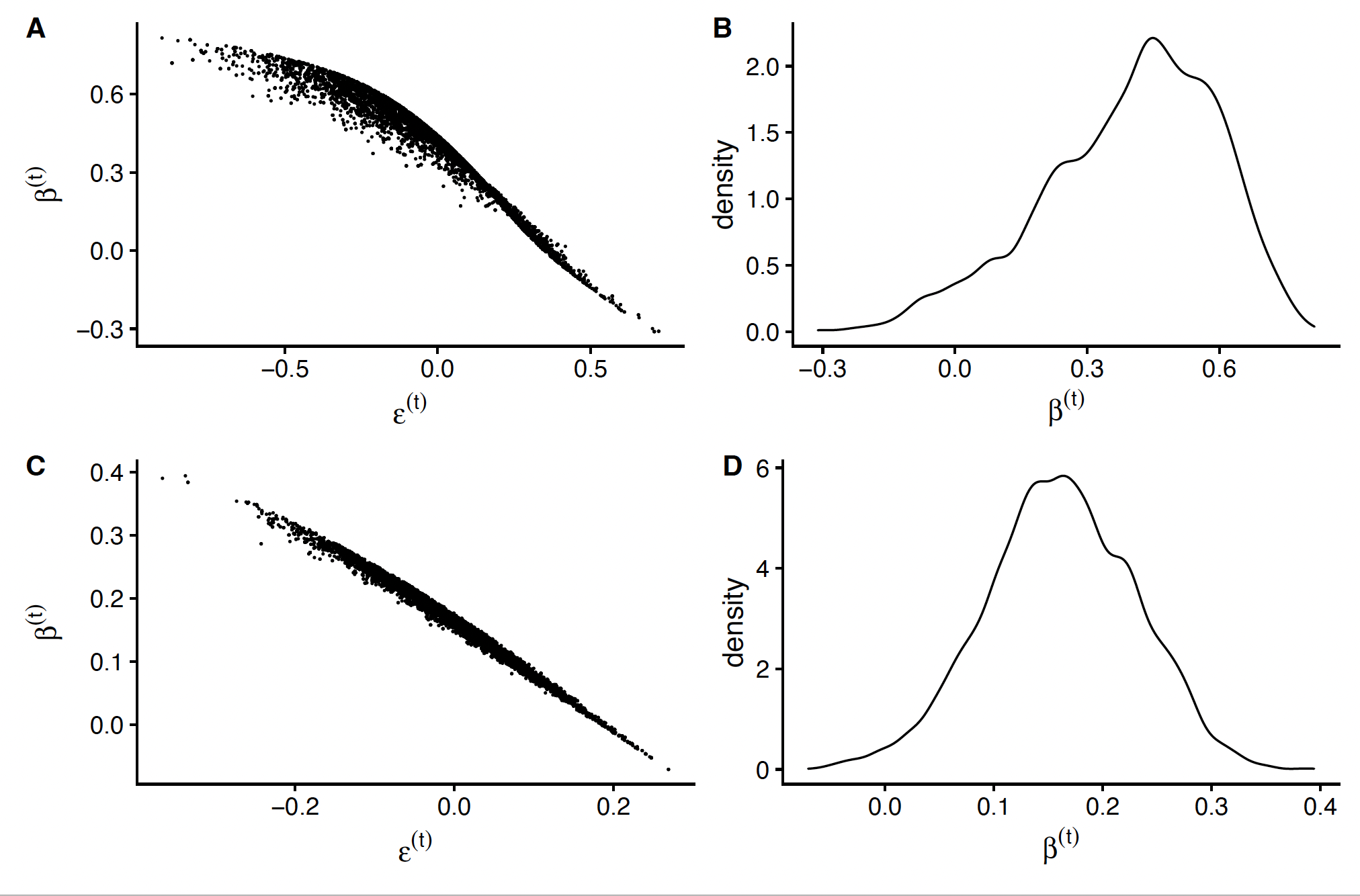}
    \caption{In finite samples, the targeted posterior can fail to accurately characterize the uncertainty in estimating $\bbeta$ because the map $\bepsilon \mapsto B(F_{\bepsilon})$ cannot reach the tails of the posterior distribution. In (A), a diagnostic plot shows joint posterior samples $\bepsilon^{(t)}$ and $\bbeta^{(t)}$, in which it can be seen that $\bbeta$ is bounded above. The marginal posterior distribution of $\bbeta$, shown in (B), is therefore skewed. Plots (C) and (D) show an example where this pathological behavior is not observed.  }
    \label{fig:tmle-diagnostic}
\end{figure}

\section{Simulation Study}
\label{section:simulation}
We present a simulation study that compares the finite sample performance of the frequentist and Bayesian targeted estimators for estimating Conditional Treatment Effects, as in the running example. All the code necessary to reproduce the simulation study is available at \url{https://github.com/herbps10/targeted_msms_paper}. The data generating distribution for the simulations is defined by the following conditional distributions:
\begin{align*}
    (X_1, X_2, X_3, X_4) &\sim N(\bm{0}, \bm{I}_4), \\
    A | X_1, X_2, X_3, X_4 &\sim \mathrm{Bernoulli}(\mathrm{logit}^{-1}(0.5X_1 - 0.5X_2 + 0.2X_3 -0.1X4)), \\
    Y | A, X_1, X_2, X_3, X_4 &\sim \mathrm{Bernoulli}(\mathrm{logit}^{-1}(X_2 + X_3 + 3A + 1.5AX_4)).
\end{align*}
A total of $500$ datasets were drawn from the data generating distribution for each sample size $n \in \{ 50, 100, 250, 500, 750, 1000 \}$.
We define a marginal structural model with linear working model \eqref{eq:linear-working-model} and squared-error loss function \eqref{eq:squared-error-loss} which yield the target parameter 
\begin{align*}
    B(P) = \argmin_{\bm{\beta} \in \mathbb{R}^2} \mathbb{E}_P\left[ \left( \Psi^{\CATEsymb}_P(X) - (1, X_4) \bm{\beta} \right)^2 \right],
\end{align*}
where $\Psi^{\CATEsymb}_P$ is defined as in the example and $\bbeta = (\beta_1, \beta_2)^\top \in \reals^2$. 

In order to investigate the behavior of the estimators under inconsistent estimation of the nuisance parameters, the parameters $g_P$ and $\bar{Q}_P$ were estimated in each sample using generalized linear models in four configurations: $g_P$ and $\bar{Q}_P$ estimated using parametric models with 
\begin{enumerate}[label={(\alph*)}]
    \item both models for $g_P$ and $\bar{Q}_P$  specified correctly,
    \item the model for $g_P$ specified correctly, that for $\bar{Q}_P$ misspecified,
    \item the model for $g_P$ misspecified, that for $\bar{Q}_P$ specified correctly,
    \item both models for $g_P$ and $\bar{Q}_P$ misspecified.
\end{enumerate}
The correctly specified regressions included all covariates, while the misspecified regressions included only $X_1$ and $X_4$ as covariates. 
For the Bayesian estimator, the following priors were used for $\beta_1$ and $\beta_2$:
\begin{align}
    (\beta_1, \beta_2) &\sim N(\bm{0}, \bm{I}_2).
\end{align}
The estimators of $\bbeta_0$ are evaluated by the empirical coverage of the 95\% credible (or confidence) intervals and in terms of absolute bias of the point estimators. For the Bayesian estimator, the posterior median is taken as a point estimator.

\paragraph*{Results} Figure  \ref{fig:simulation_coverage} shows the empirical coverage of the 95\% credible (confidence) intervals, and Figure~\ref{fig:simulation_bias} shows the absolute bias. The results are also summarized as Table \ref{tab:simulation_results} in the appendix. Both the estimators achieve optimal empirical coverage of the 95\% credible (confidence) interval and achieve low absolute bias when both nuisance parameters are correctly specified. Of note, for the smallest sample size the Bayesian estimator achieves slightly better coverage than the frequentist estimator, suggesting the Bayesian approach may be more conservative in some finite sample settings.
When either of the estimators are estimated with misspecified regressions the absolute bias still approaches zero as the sample size increases, demonstrating the double-robust properties of the estimators. Interestingly, the 95\% empirical coverage is not greatly affected by the misspecification, a result that is not guaranteed theoretically. The good performance in this case may be caused by the specific setup of the simulation study. Finally,  when both nuisance parameters are estimated inconsistently the absolute bias is larger than in the other scenarios and the empirical coverage is highly degraded for $\beta_1$. The fact that the empirical coverage of $\beta_2$ remains near-optimal is likely a result of the simulation setup.

\begin{figure}
    \centering
    \includegraphics[width=0.7\textwidth]{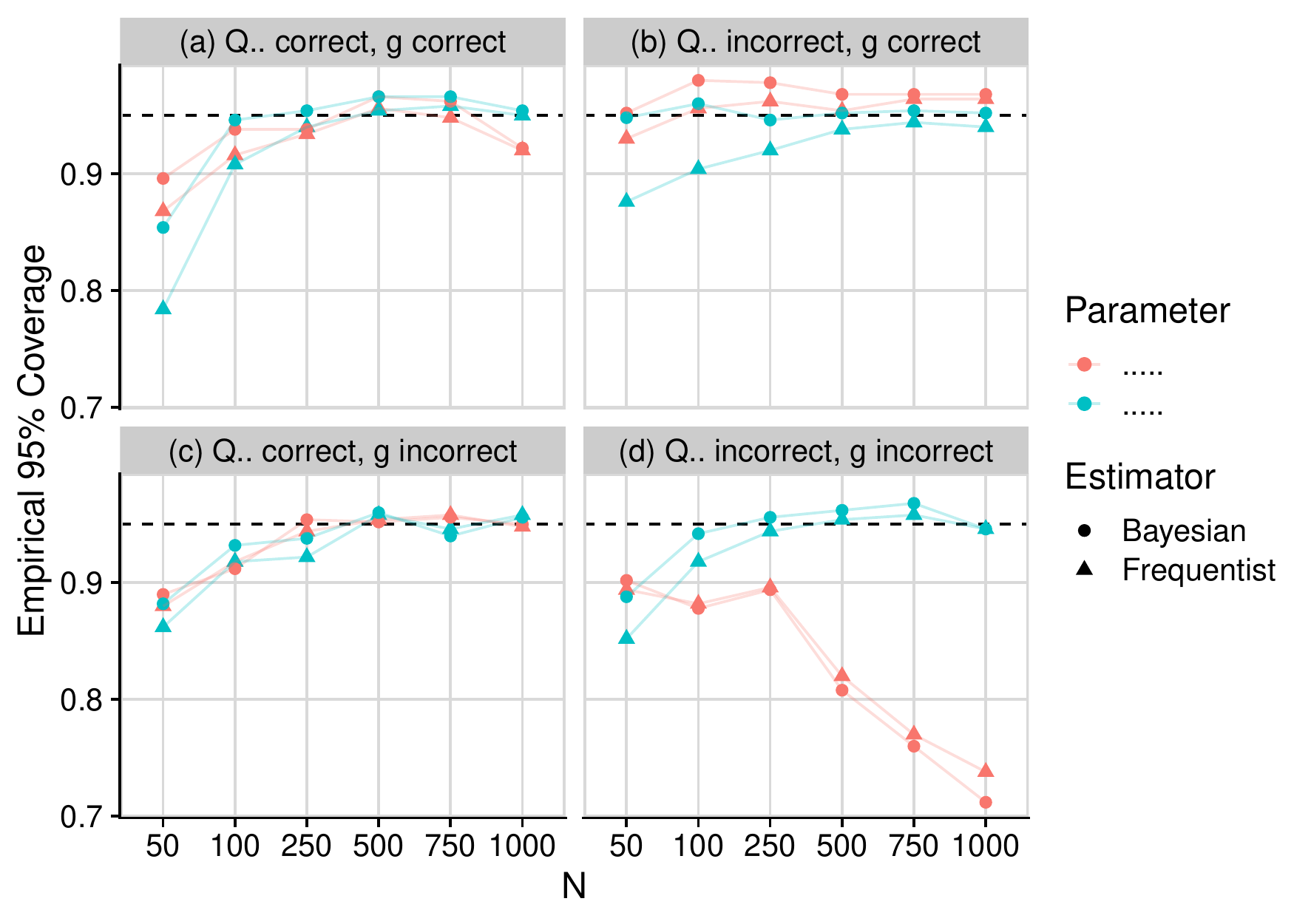}
    \caption{Frequentist empirical coverage of 95\% credible and confidence intervals in the simulation study for scenarios (a), (b), (c), and (d).}
    \label{fig:simulation_coverage}
\end{figure}

\begin{figure}
    \centering
    \includegraphics[width=0.7\textwidth]{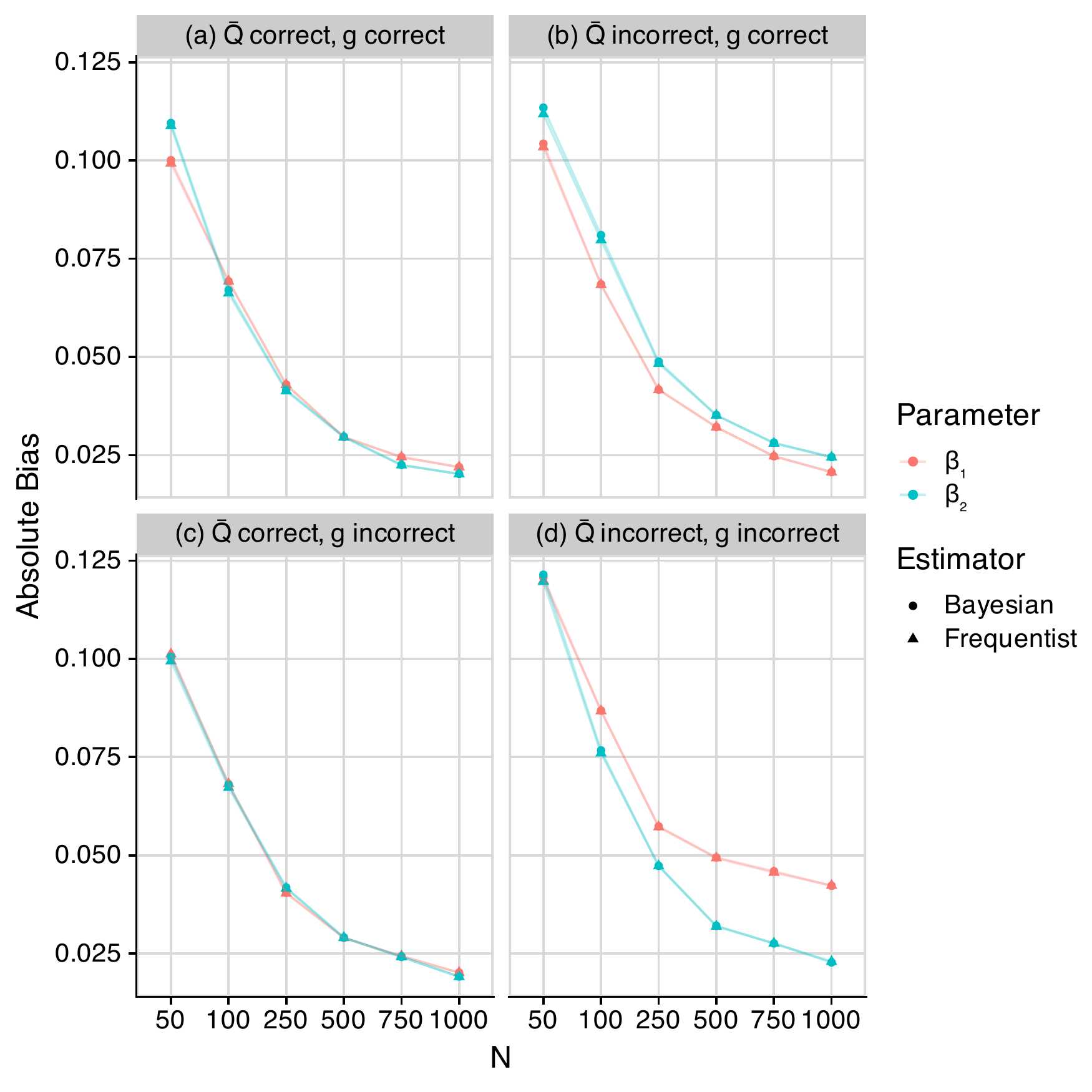}
    \caption{Absolute bias of estimator in the simulation study for scenarios (a), (b), (c), and (d).}
    \label{fig:simulation_bias}
\end{figure}

\section{Application}
\label{section:motivating-application}
As an example, we define and estimate the causal effect of a broad-based family planning intervention on postpartum contraceptive use using data from a randomized field experiment conducted in Lilongwe, Malawi \citep{karra2022fptrial}. All the code necessary to reproduce the data analysis is available at \url{https://github.com/herbps10/targeted_msms_paper}.

Postpartum women were randomized into intervention and control groups in which intervention consisted of receiving a family planning information package and counseling sessions, free transportation to a family planning clinic, free family planning services at the clinic, and phone consultations with a doctor and reimbursement for any treatment necessary due to side effects. Contraceptive use was measured as an outcome after a 2-year interval. 

We use the publicly available replication dataset from the original paper to conduct our analysis \citep{karra2022fpdata}. We refer to the paper and published study protocol for detailed information regarding the study design and data processing \citep{karra2020effect, karra2022fptrial}. The variables for each participant used in our analysis are:
\begin{itemize}
    \item $X$: number of children who are alive, educational attainment (primary or less vs. secondary or higher), age (three bins), age of sexual debut, ever used family planning, religion, work status, tribal group, neighborhood, and current contraceptive use, all measured at baseline;
    \item $A$: binary indicator of inclusion in intervention group;
    \item $Y$: binary indicator of contraceptive use at endline.
\end{itemize}
Following the original analysis, 5 observations were excluded due to missing data, leaving an analysis dataset of $1667$ observations. The intervention group included $781$ participants ($46.9\%$). At endline, $1240$ participants were recorded as using contraceptives ($74.3\%$). The number of children alive at baseline (from now on referred to as ``number of children at baseline"), which we consider as a potential treatment effect modifier, ranged from $0$ to $9$ (mean: 2.3, standard deviation: 1.3). Figure \ref{fig:pnas_effect_modifier} shows the full distribution of children at baseline.

\begin{figure}
    \centering
    \includegraphics[width=0.9\textwidth]{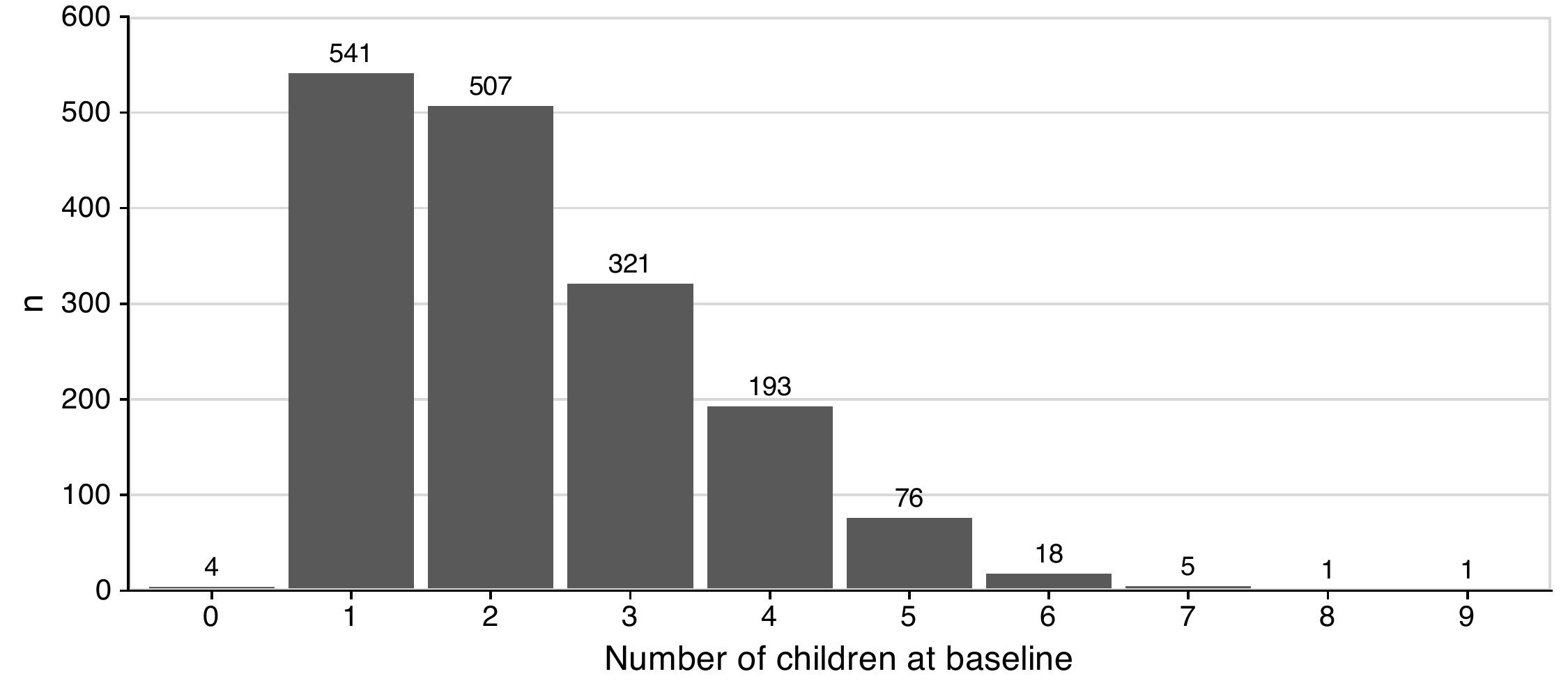}
    \caption{Distribution of the number of children at baseline among the cohort of women in the example.}
    \label{fig:pnas_effect_modifier}
\end{figure}

\paragraph*{Target Parameter} Let $\Psi_P^{\CATEsymb}$ be the Conditional Average Treatment Effect as defined in the running example, interpretable as the expected difference in contraceptive use at endline conditional on inclusion vs. non-inclusion in the intervention group. We consider the total number of children who are alive at baseline, denoted $X_{\mathrm{children}}$, to be a potential treatment effect modifier, and set $V =  X_{\mathrm{children}}$. A linear working model and squared-error loss function are adopted, as in the example. The target parameter is therefore 
\begin{align}
    B^\CATEsymb(P) = \argmin_{\bbeta \in \reals^2} \E_P\left[ \left( \PsiP^\CATEsymb(X) - \bbeta^\top (1, X_{\mathrm{children}})^\top\right)^2 \right].
\end{align}
The MSM parameters $\beta_1$ and $\beta_2$ can be interpreted as the intercept and slope of a line that best summarizes (in terms of a squared-error loss) the relationship between the CATE and number of children at baseline.

\paragraph*{Causal Identification} The statistical parameter $\bbeta_0$ can be given a causal interpretation if the identification assumptions for $\Psi^{\CATEsymb}$ hold. Positivity ($g_P(1, X) > 0$ for all $X$) in this setting can be interpreted as the assumption that the probability of randomization into the intervention group is  positive for all strata of covariates. The no unmeasured confounders assumption requires that the potential outcomes $Y(0)$ and $Y(1)$ are independent of intervention status conditional on the covariates. Both assumptions reasonably hold given the randomized study design. It is possible that the consistency assumption, however, may not hold due to spillover effects between the intervention and control groups.

\paragraph*{Estimation} The nuisance parameters $g_P$ and $\bar{Q}_P$ were estimated using an ensemble of learners using the SuperLearner algorithm \citep{vanderLaan2007superlearner, polley2021superlearner}. The base learners included generalized linear models, $\ell_1$-penalized regression \citep{friedman2010glmnet,simon2011glmnet}, random forests \citep{wright2017ranger}, gradient boosting trees \citep{chen2022xgboost}, and the Highly Adaptive Lasso \citep{benkeser2016highly,coyle2022hal,hejazi2020hal}. For the Bayesian estimator, the following weakly informative priors were applied, reflecting a lack of strong prior knowledge of the true values of the parameters:
\begin{align}
    (\beta_1, \beta_2) &\sim N(\bzero, \bm{I}_2).
\end{align}
The Metropolis-Hastings algorithm was run for $50,000$ iterations to yield a set of joint draws from the posterior distribution.

\paragraph*{Results} The frequentist estimates for $\beta_1$ and $\beta_2$ were $-0.016$ (95\% CI: $[-0.1, 0.07]$) and $0.033$ (95\% CI: $[0.0019, 0.063]$), respectively. The two-sided hypothesis test of $H_0: ``\beta_2 = 0"$ vs. $H_1: ``\beta_2 \neq 0"$ yields a $p$-value of $p=0.038$. Thus there is sufficient evidence at the $\alpha = 0.05$ significance level to reject the null hypothesis that $\beta_2 = 0$. The Bayesian posterior medians for $\beta_1$ and $\beta_2$ were $-0.015$ (95\% CI: $[-0.1, 0.074]$) and $0.032$ (95\% CI: $[-0.001, 0.066]$), respectively. The credible intervals can be interpreted as the regions in which the true values of $\beta_{1}$ and $\beta_2$ reside with 95\% probability. The posterior probability that $\beta_1 > 0$ is $97.1\%$. The posterior distribution of the working model $\{ m_{\bbeta} : \bbeta \in \mathcal{B} \}$ is shown in Figure \ref{fig:pnas_posterior_summary}. Both the frequentist and Bayesian methods reach substantially the same conclusion, although care should be taken to interpret the results following the correct interpretation of probability for each approach.

\begin{figure}
    \centering
    \includegraphics[width=0.8\textwidth]{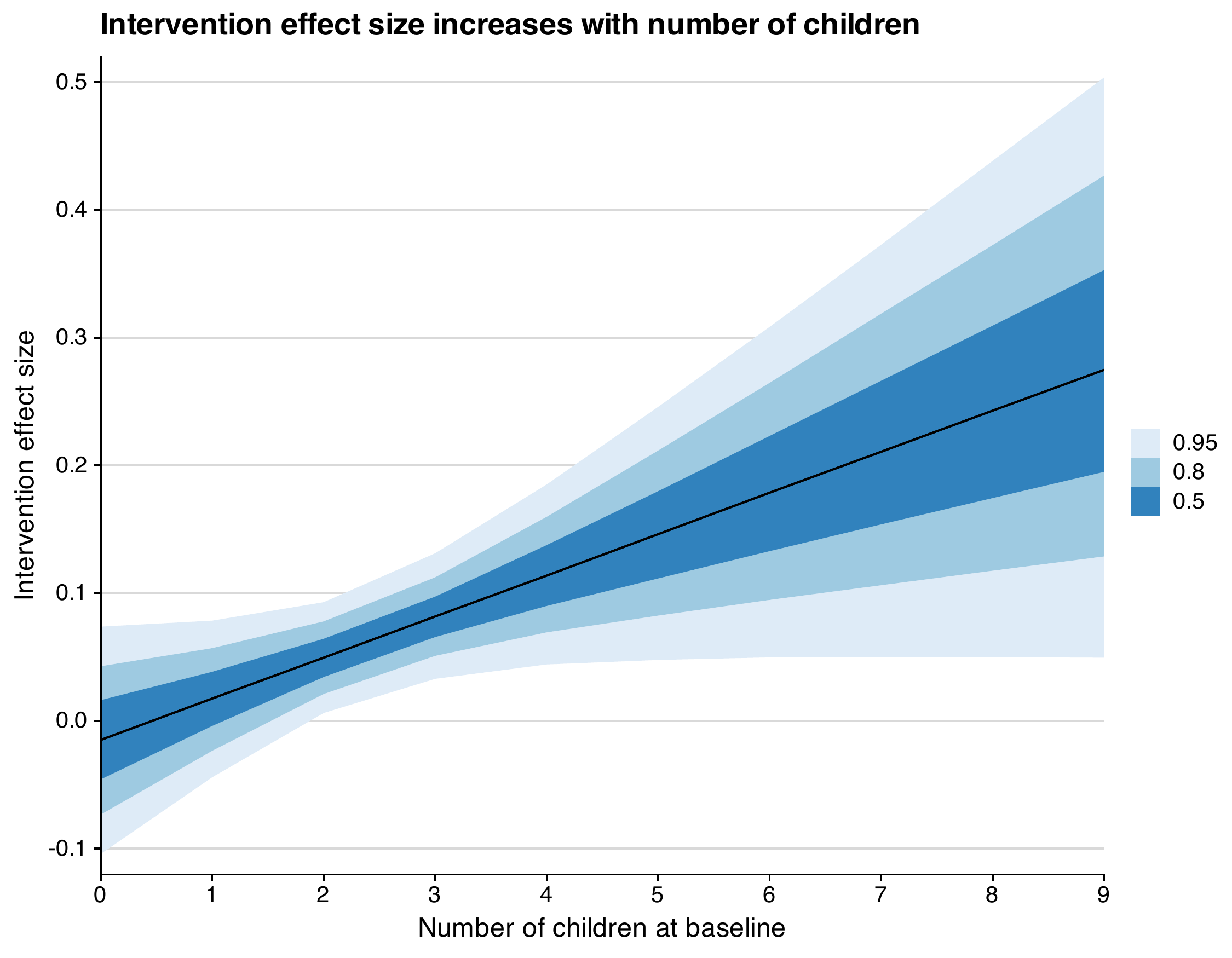}
    \caption{Posterior distribution of the parameter $\bbeta$ of the  working model $\{ m_{\bbeta} : \bbeta \in \mathcal{B} \}$, which can be interpreted as a linear approximation under a squared error loss of the conditional average treatment effect of the intervention on contraceptive use at endline. The black line is the posterior median and the shaded regions are the 50\%, 80\%, and 95\% credible intervals. The results suggest the treatment effect grows as the number of children increases.}
    \label{fig:pnas_posterior_summary}
\end{figure}

Kernel density plots of the posterior distribution of the MSM parameters are shown in Figure \ref{fig:pnas_posterior}. The figure also includes normal densities centered on the targeted point estimate $\bbeta_n^*$ and with variance given by the variance of $D^*(P_n^*)$, the estimate of the efficient influence function. As expected under the Bernstein von-Mises theorem (Theorem \ref{theorem:bvm:full}), the posterior distribution and normal densities are similar.

\begin{figure}
    \centering
    \includegraphics[width=0.9\textwidth]{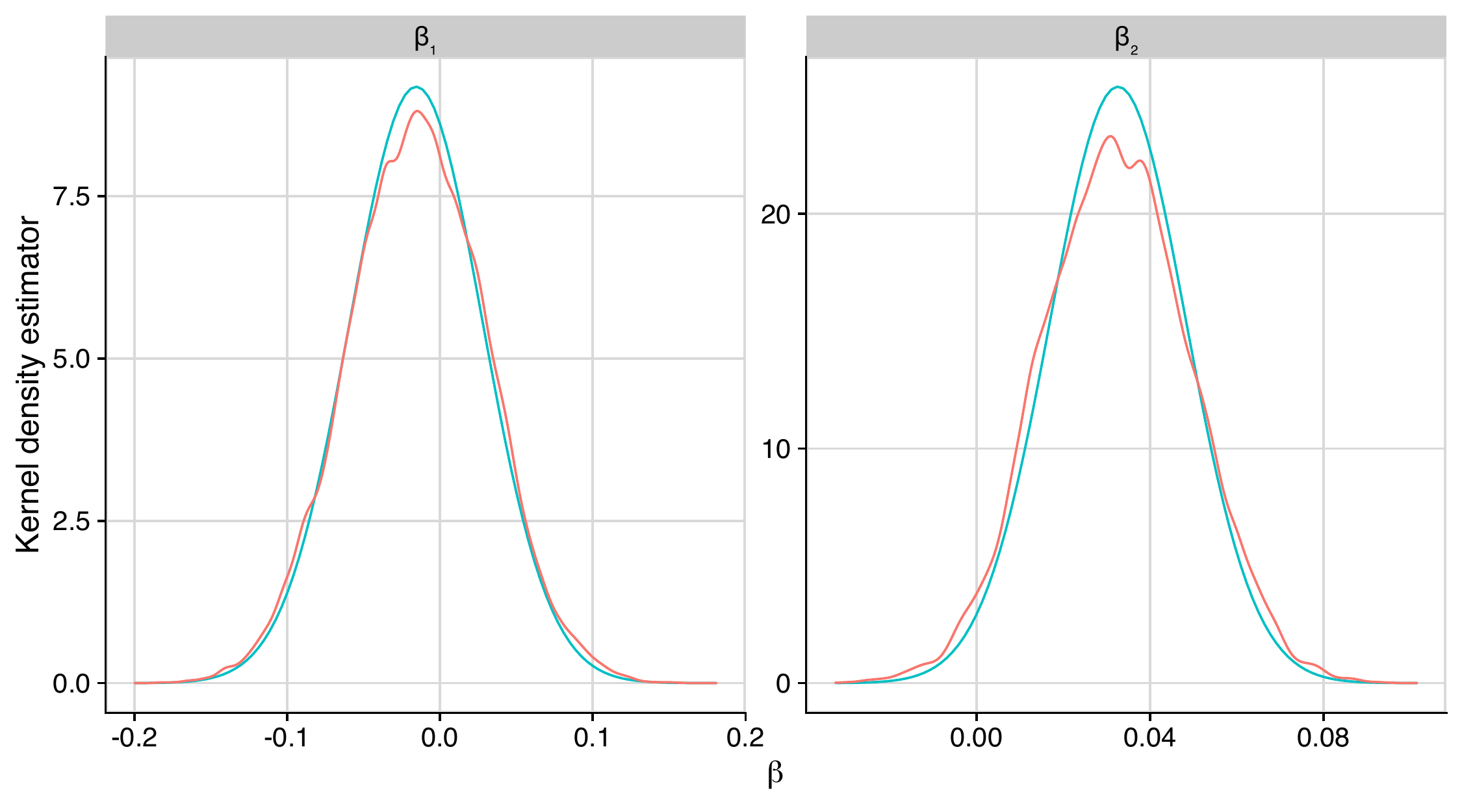}
    \caption{Kernel density plot of posterior distribution (red line) of $\beta_1$ and $\beta_2$ from the example application. A normal density is also pictured (blue line) centered on the frequentist point estimates $\hat{\bbeta}_n^*$ with variance given by the targeted estimate of the variance of the efficient influence function $D^*(P_0)$.}
    \label{fig:pnas_posterior}
\end{figure}

\section{Discussion}
\label{section:discussion}
In this work, we proposed a novel targeted Bayesian estimator for the parameter of a Marginal Structural Model. The general notation for MSMs defines a  parameter $B(P)$ which is the parameter of a user-specified working model $\{ m_{\bbeta} : \bbeta \in \mathcal{B} \} $ that best approximates a functional summary $\Psi_P$ with respect to a loss function $L$. A semi-parametric analysis of $P \mapsto B(P)$ derived its (already known) efficient influence function $D^*(P)$ at any $P$ in the model. The variance  $\mathrm{Var}_{P}(D^*(P)(O))$ is the asymptotic efficiency bound for estimating $B(P)$ within a non-parametric model. We then presented a general framework for efficient estimation using Targeted Maximum Loss-Based Estimation, and introduced a novel targeted Bayesian estimator. 

Bayesian TMLE blends frequentist and Bayesian approaches to statistical inference in a way that enjoys benefits of both approaches. From the frequentist TMLE it inherits a non-parametric estimation strategy that avoids unrealistic assumptions on the data generating process. Flexible estimation methods, including algorithms from machine learning, can be used (or combined, via ensemble methods) to estimate the required nuisance parameters, and are then naturally incorporated into the targeted Bayesian analysis. From the Bayesian paradigm, the targeted method inherits the ability to incorporate prior information into the analysis and the subjective interpretation of probability. An additional benefit is that contrary to traditional Bayesian non-parametric inference, which generally requires placing priors on complex probability spaces or abstract hyperparameters, in the targeted approach priors can be placed directly on the parameter of interest. For MSMs, priors are placed directly on $B(P)$.

We stress that the Bayesian TMLE should not be interpreted as ``100\% Bayesian" because Bayesian inference is not used to model the entire data-generating process (that is, the full joint distribution of $X$ and $Z$ is not modeled). Indeed, estimates of the required nuisance parameters are estimated separately and not necessarily within a Bayesian paradigm. In other words, the method is not ``generative", as is typically the case for Bayesian models \citep{gelman2020bayesian}. Bayesian TMLE is thus better understood as a hybrid approach that combines aspects of frequentist and Bayesian paradigms. 

Our work leaves open multiple directions for future research. More investigation is warranted to understand the finite sample performance of both the frequentist and Bayesian estimators for more complex data generating distributions and working models of high dimension.  In addition, further theoretical work is necessary to establish that the targeted posterior  converges to a distribution centered on the frequentist MLE with optimal variance given by the variance of the EIF. Part of the challenge with this result is to show it holds under reasonable assumptions on the convergence of the nuisance parameters, such as that $\norm{\bar{Q}_n^* - \bar{Q}_0}_2 = o_p(n^{-1/4})$. 

Bayesian inference allows for the introduction of prior information in a principled manner. While in the analyses presented here we used independent and weakly informative priors, an interesting direction for future work would be to investigate more complex priors. Hierarchical priors are commonly used in applied Bayesian analyses to share information between units, such that parameters are partially pooled towards a common mean \citep{gelman2013bda3}. This has the benefit of stabilizing estimates, particularly in small samples. Hierarchical priors could be applied to MSM parameters that exhibit hierarchical structure, which we hypothesize may lead to improved finite sample performance. As long as such hierarchical priors satisfy the (weak) assumptions necessary for the Bernstein von-Mises theorem, most importantly that the true parameter is included in the support of the prior, then such a prior choice will not affect the desirable asymptotic properties of the targeted posterior distribution.

\section{Acknowledgements}
We would like to thank Ismaël Castillo for helpful discussions, and Leontine Alkema and Laura Balzer for useful feedback. This article was written while the first author was a visiting researcher at MAP5 (UMR CNRS 8145, Université Paris Cité). This material is based upon research supported by the Chateaubriand Fellowship of the Office for Science \& Technology of the Embassy of France in the United States.

\bibliography{bibliography}

\section{Appendix}

\subsection{Full Statement and Proof of Theorem \ref{theorem:general-eif:full}
}
\label{proof:general-eif}

Here is the full version of Result~\ref{theorem:general-eif:short}. Recall that, for any $P\in\model$ and $s \in L^2_0(P)$ such that $s \neq 0$, $\norm{s}_\infty < \infty$, $\left\{P_{\epsilon} : |\epsilon| < \norm{s}_\infty^{-1} \right\} \subset \model$ fluctuates $P$ in the direction of $s$. For clarity, we now introduce the mapping $\Beps:\epsilon \mapsto B(P_{\epsilon})$.

\begin{theorem}[Efficient Influence Function of $B$]
\label{theorem:general-eif:full}
Let us assume that:
\begin{enumerate}
    \item \label{assumption:derivatives-exist} For every $P \in \model$, for every $t \in \mathcal{T}$, the following two conditions are met $P$-almost surely:
        \begin{enumerate}
            \item  $\bbeta' \mapsto L_m(t, \bbeta')(X)$ is differentiable at every $\bbeta \in \mathcal{B}$ with derivative $\dot{L}_m(t, \bbeta)(X) \in \reals^p$,
            \item $\bbeta' \mapsto \dot{L}_m(t, \bbeta')(X)$ is differentiable at every $\bbeta \in \mathcal{B}$ with derivative $\ddot{L}_m(t, \bbeta)(X) \in \reals^{p \times p}$.
        \end{enumerate}
        Moreover, for every $P \in \model$, for every $\bbeta \in \mathcal{B}$, it holds $P$-almost surely that 
        \begin{enumerate}\setcounter{enumii}{2}
            \item $t' \mapsto \dot{L}_m(t', \bbeta)(X)$ is differentiable at every $t \in \mathcal{T}$ with derivative $\gradient{} \dot{L}_m(t, \bbeta)(X) \in \reals^p$.
        \end{enumerate}
    \item \label{assumption:implicit-function-theorem} The function $(\epsilon, \bbeta) \mapsto U(\epsilon, \bbeta) := \E_{\Peps}\left[ \dot{L}_m(\Psi_{\Peps}(X), \bbeta)(X) \right]$ is differentiable in the neighborhood of $(0, \Beps(0))$. 
    \item \label{assumption:matrix-invertible} For every $P \in \model$, $\E_{P}\left[ \ddot{L}_m(\Psi_{P}(X), B(P))(X) \right]$ is invertible.
    \item \label{assumption:u-equals-zero} For every $P \in \model$, for all $s \in L^2_0(P)$ such that $s \neq 0$, $\norm{s}_\infty < \infty$, it holds that $U(\epsilon, \Beps(\epsilon)) = 0$ in an $\epsilon$-neighborhood of $0$.
    \item (Dominated Convergence) \label{assumption:bounded-hessian} For all $P \in \mathcal{M}$, there exists a random variable $G_1 \in \reals_{+}^{p \times p}$ with $\E_P[G_1] < \infty$, such that $|\ddot{L}_m(\Psi_P(X), \bbeta)(X)| \leq G_1$ $P$-almost surely in a $\bbeta$-neighborhood of $\Beps(\epsilon)$. Here and below the absolute value and inequalities are to be understood entrywise.
    \item (Dominated Convergence) \label{assumption:bounded-derivative} For all $P \in \mathcal{M}$, there exists a random variable $G_2 \in \reals_+^p$ with $\E_P[G_2] < \infty$ such that $|\frac{d}{d\epsilon} \dot{L}_m(\Psi_{P_\epsilon}(X), \Beps(\epsilon))(X)| \leq G_2$ $P$-almost surely in an $\epsilon$-neighborhood of $0$.
    \item \label{assumption:delta} For all $P \in \model$, for all $s \in L^2_0(P)$ such that $s \neq 0$, $\norm{s}_\infty < \infty$, if $\left\{ P_\epsilon : |\epsilon| < \norm{s}_\infty^{-1} \right\} \subset \model$ is characterized by $\frac{dP_\epsilon}{dP} = 1 + \epsilon s$ for all $\epsilon \in \reals$ such that $|\epsilon| < \norm{s}_\infty^{-1}$, 
    then the mapping $\epsilon \mapsto \Psi_{P_\epsilon}(X)$ is almost surely differentiable at 0 and
    \begin{align}
        \frac{d}{d\epsilon} \Psi_{\Peps}(X) \evalat{\epsilon = 0} = \E_P \left[\Delta^*(P)(O) s(O) \mid X \right]
    \end{align}
    for a function $\Delta^*(P) \in L_0^2(P)$. 
\end{enumerate}
Then the target functional $P \mapsto B(P)$ is pathwise differentiable at every $P \in \model$, with an efficient influence function $D^*(P)$ at $P$  given by
\begin{align}
    D^*(P)(O) &= M^{-1} \left[
        D_1^*(P)(O) + D_2^*(P)(X)
    \right],
\end{align}
where $D_1^*(P), D_2^*(P) \in L_0^2(P)$ are given by
\begin{align}
    D_1^*(P)(O) &= \gradient{} \dot{L}_m(\Psi_P(X), B(P))(X) \times \Delta^*(P)(O), \\
    D_2^*(P)(X) &= \dot{L}_m(\Psi_P(X), B(P))(X),
\end{align}
and the normalizing matrix $M$ is given by
\begin{align}
    M = -\mathbb{E}_{P}\left[\ddot{L}_m(\PsiP(X), B(P))(X) \right].
\end{align}
\end{theorem}

\begin{proof}
Set arbitrarily $P \in \model$ and $s \in L_0^2(P)$ such that $s \neq 0$ and $\norm{s}_\infty < \infty$. Consider the submodel $\left\{ P_\epsilon : |\epsilon| < \norm{s}_\infty^{-1} \right\}$ given by $\frac{dP_\epsilon}{dP} = 1 + \epsilon s$. By assumption \ref{assumption:u-equals-zero}, $U(\epsilon, \Beps(\epsilon)) = 0$ in an $\epsilon$-neighborhood of~$0$. %The Implicit Function Theorem holds under assumptions \ref{assumption:implicit-function-theorem} and \ref{assumption:matrix-invertible}. under which
By assumption \ref{assumption:implicit-function-theorem}, $U$ is differentiable in a neighborhood of $(0, \Beps(0))$. Denote by $\frac{d U}{d \epsilon}$ and $\frac{\partial U}{\partial \bbeta}$ the derivative and gradient of $U$ with respect to its $\epsilon$ and $\bbeta$ arguments, respectively. Differentiating $\epsilon \mapsto U(\epsilon, \Beps(\epsilon))$ at $\epsilon = 0$ yields
\begin{align}
    \frac{d U}{d \epsilon} (0, \Beps(0)) +  \frac{\partial U}{\partial \bbeta}(0, \Beps(0)) \frac{d \Beps}{d \epsilon} (0) = 0.
\end{align}
Solving for $\frac{d \Beps}{d \epsilon}(0)$, and using assumption \ref{assumption:matrix-invertible}, yields
\begin{align}
    \label{eq:dB-depsilon}
    \frac{d \Beps}{d \epsilon}(0) = - \left\{ \frac{\partial U}{\partial \bbeta}(0, \Beps(0)) \right\}^{-1} \frac{d U}{d \epsilon}(0, \Beps(0)).
\end{align}
Now we compute each of these derivatives. The first derivative becomes (the second equality is justified below)
\begin{align*}
    -\frac{\partial U}{\partial \bbeta} (0, \Beps(0)) &= -\frac{\partial}{\partial \bbeta} \E_{\Peps}\left[ \dot{L}_m(\Psi_{\Peps}(X), \bbeta)(X) \right] \evalat{\epsilon = 0, \bbeta = \Beps(0)} \\
                                   &= -\E_{\Peps}\left[ \frac{\partial}{\partial \bbeta} \dot{L}_m(\Psi_{\Peps}(X), \bbeta)(X) \evalat{\epsilon = 0, \bbeta = \Beps(0)} \right] \\
                                     &= -\E_{P}\left[ \ddot{L}_m(\Psi_{P}(X), \Beps(0))(X) \right] \\
                                     &=: M,
\end{align*}
where in the second line the gradient and expectation can be exchanged under the Dominated Convergence theorem (DCT) and assumptions \ref{assumption:derivatives-exist} and \ref{assumption:bounded-hessian}. Note that $M$ is invertible under assumption~\ref{assumption:matrix-invertible}. The second derivative splits into two parts, first by definition of $U$, second by definition of $P_\epsilon$:
\begin{align}
    \frac{d U}{d \epsilon}(0, \Beps(0)) =& \frac{d}{d\epsilon} \E_{\Peps}\left[ \dot{L}_m(\Psi_{\Peps}(X), \bbeta)(X) \right] \evalepszero \\
                                        =& \frac{d}{d\epsilon} \E_{P} \left[  \dot{L}_m(\Psi_{\Peps}(X), \bbeta)(X)  \right] \evalepszero \\
                                        &+ \frac{d}{d\epsilon} \epsilon \E_{P}\left[  \dot{L}_m(\Psi_{\Peps}(X), \bbeta)(X) s(O) \right] \evalepszero \label{eq-dU-depsilon}.
\end{align}
The first term becomes (the first equality is justified below)
\begin{align}
    \frac{d}{d\epsilon} \E_{P} \left[ \dot{L}_m(\Psi_{\Peps}(X), \bbeta)(X) \right] \evalepszero &=  \E_{P} \left[ \frac{d}{d\epsilon} \dot{L}_m(\Psi_{\Peps}(X), \bbeta)(X) \evalepszero \right]\\
    &= \E_P \left\{ \gradient{} \dot{L}_m(\Psi_{\Peps}(X), \bbeta)(X) \evalepszero \frac{d}{d\epsilon} \Psi_{\Peps}(X) \evalat{\epsilon = 0} \right\} \\
    &= \E_P \left\{ \gradient{} \dot{L}_m(\Psi_{P}(X), \Beps(0))(X) \times \Delta^*(P)(O) s(O) \right\} \label{eq:depsilon-part1},
\end{align}
where the interchange of the derivative and expectation in the first line is by the DCT and assumptions~\ref{assumption:derivatives-exist} and \ref{assumption:bounded-derivative}. 
For the second term, apply the product rule to yield
\begin{align}
    \frac{d}{d\epsilon} \epsilon \E_{P}\left[ \dot{L}_m(\Psi_{\Peps}(X), \bbeta)(X) s(O) \right] \evalepszero =& \left[ \epsilon \frac{d}{d\epsilon} \E_P \left[ \dot{L}_m(\Psi_{\Peps}(X), \bbeta)(X) s(O) \right] \right] \evalepszero \\
    &+ \E_P\left[ \dot{L}_m(\Psi_{\Peps}(X), \bbeta)(X) \evalepszero s(O) \right] \\
    =& \E_P\left[ \dot{L}_m(\Psi_P(X), \Beps(0))(X) s(O) \right] \label{eq:depsilon-part2}.
\end{align}
Collecting \eqref{eq:dB-depsilon}, \eqref{eq-dU-depsilon}, \eqref{eq:depsilon-part1}, and \eqref{eq:depsilon-part2} yields
\begin{align}
    \frac{d\Beps}{d\epsilon}(0) =& M^{-1} \bigg[  \E_P \left[ \gradient{}\dot{L}_m(\Psi_{P}(X), \Beps(0))(X) \times \Delta^*(P)(O) s(O) \right] \\
    &+  \E_P\left[ \dot{L}_m(\Psi_P(X), \Beps(0))(X) s(O) \right] \bigg],
\end{align}
hence
\begin{align}
    D^*(P)(O) &= M^{-1} \left[
        \gradient{} \dot{L}_m(\Psi_P(X), B(P))(X) \times \Delta^*(P)(O) +
        \dot{L}_m(\Psi_P(X), B(P))(X)
    \right].
\end{align}
This completes the proof.
\end{proof}

\subsection{Proof of Lemma \ref{lemma:eif-delta-example} }
\label{proof:eif-delta-example}
\begin{proof}
The derivative of $\Psi_{P_\epsilon}$ evaluated at $\epsilon = 0$ decomposes into two parts:
\begin{align}
    \frac{d}{d\epsilon} \Psi_{\Peps}(X) \evalat{\epsilon = 0} &= \frac{d}{d\epsilon} \bar{Q}_{P_\epsilon}^{(1)}(X) -  \bar{Q}_{P_\epsilon}^{(0)}(X) \evalat{\epsilon = 0}.
\end{align}
For any $a \in\{0,1\}$, write $\bar{Q}_{P_\epsilon}^{(a)}(X)$ and apply the definition of $P_\epsilon$ to arrive at
\begin{align}
    \bar{Q}_{P_\epsilon}^{(a)}(X) &= \E_{\Peps}\left[ Y | A = a, X \right] \\
    &= \frac{\E_P\left[ Y | A = a, X \right] + \epsilon \E_P\left[Y s(O) | A = a, X \right]}{1 + \epsilon \E_P[s(O) | A = a, X]} \text{\quad \citep[Lemma 1]{achambaz2012importance}} \\
    &= \frac{\E_P[Y | A = a, X] + \epsilon \E_P[s(O)(Y - \E_P[Y | A = a, X]) | A = a, X]}{1 + \epsilon \E_P[s(O) | A = a, X]} \\
        & \qquad + \frac{ \epsilon \E_P[Y | A = a, X] \E_P[s(O) | A = a, X]}{1 + \epsilon \E_P[s(O) | A = a, X]} \\
    &= \frac{\E_P[Y | A = a, X](1 + \epsilon \E_P[s(O) | A = a, X]) + \epsilon \E_P[s(O)(Y - \E_P[Y | A = a, X]] | A = a, X]}{1 + \epsilon \E_P[s(O) | A = a, X]} \\
    &= \E_P[Y | A = a, X] + \frac{\epsilon \E_P[s(O)(Y - \E_P[Y | A = a, X]] | A = a, X]}{1 + \epsilon \E_P[s(O) | A = a, X]}.
\end{align}
Taking the derivative of $\bar{Q}_{P_\epsilon}^{(a)}$ at $\epsilon = 0$ yields
\begin{align}
    \frac{d}{d \epsilon} \bar{Q}_{P_\epsilon}^{(a)}(X) \evalat{\epsilon = 0} &= \E_P\left[s(O)(Y - \E_P[Y | A = a, X]) \middle| A = a, X\right] \\
    &= \E_P\left[s(O)\frac{\mathbb{I}(A = a)}{P[A = a \mid X]}\left(Y - \E_P[Y | A = a, X]\right) \middle| X \right] \\
    &= \E_P\left[s(O)\frac{\mathbb{I}(A = a)}{P[A = a \mid X]}(Y - \bar{Q}_{P}^{(A)}(X) ) \middle| X \right].
\end{align}
Therefore
\begin{align}
    \frac{d}{d\epsilon} \Psi_{\Peps}(X) \evalat{\epsilon = 0} &= \E_P\left[ s(O) \left\{ \frac{\mathbb{I}(A = 1)}{P\left[A = 1 \middle| X\right]} -  \frac{\mathbb{I}(A = 0)}{P\left[A = 0 \middle| X\right]}  \right\} (Y - \bar{Q}_P^{(A)}(X)) \middle| X \right],
\end{align}
which completes the proof.
\end{proof}

\subsection{Proof of Theorem \ref{theorem:eif-example}}
\label{proof:eif-example}
\begin{proof}
The result follows directly from Theorem \ref{theorem:general-eif:full}. We start by checking each of the assumptions required therein.
For every $P\in\model$, $t\in\mathcal{T}$, the derivatives $\dot{L}_m(t, \bbeta)(X) = -2(t - \bbeta^\top (1,V)^{\top})(1,V)^{\top}$ and $\ddot{L}_m(t, \bbeta)(X) = 2(1,V)^\top (1,V)$ exist at every $\bbeta \in \mathcal{B}$ $P$-almost surely. For every $P \in \model$, $\bbeta \in \mathcal{B}$, the derivative $\gradient{} \dot{L}_m(t, \bbeta)(X) = -2(1,V)^{\top}$ exists at every $t \in \mathcal{T}$ $P$-almost surely. Therefore Assumption (\ref{assumption:derivatives-exist}) is satisfied. Assumptions (\ref{assumption:implicit-function-theorem})-(\ref{assumption:bounded-derivative}) are satisfied by assumption. Lemma (\ref{lemma:eif-delta-example}) shows that Assumption $(\ref{assumption:delta})$ holds, and gives the form of $\Delta^*(P)$. Plugging in the calculated derivatives $\ddot{L}_m, \ddot{L}_m, \gradient{}\dot{L}_m$ into the form of $D^*(P)$ of Theorem~\ref{theorem:general-eif:full} and simplifying gives the stated result. 
\end{proof}

\subsection{Proof of Theorem \ref{theorem:asymptotic-normality}}
\label{proof:tmle-asymptotic-normality}
\begin{proof}
To begin, write
\begin{align}
    \hat{\bbeta}^*_n - \bbeta_0 =& -P_n[D^*(P_n^*)] & \text{ (bias term)} \\
    &+ (P_n - P_0)[D^*(P_0)] & \text{ (CLT term)} \\
    &+ (P_n - P_0)[D^*(P_n^*) - D^*(P_0)] & \text{ (empirical process term)} \\
    &+ \hat{\bbeta}_n^* - \bbeta_0 + P_0[D^*(P_n^*)] & \text{ (second-order remainder)}. 
\end{align}
We take each of these terms in turn:
\begin{description}
    \item[bias term:] $P_n [D^*(P_n^*)] = o_P(n^{-1/2})$ by construction of the TMLE fluctuation; 
    \item[CLT term:] $\sqrt{n}(P_n - P_0)[D^*(P_0)] \rightsquigarrow N(0, P_0[D^*(P_0) D^*(P_0)^\top)]$ by the Central Limit theorem;
    \item[empirical process term:]  $(P_n - P_0)[D^*(P_n^*) - D^*(P_0)] = o_P(n^{-1/2})$ by the assumed Donsker conditions \citep[Lemma 19.24]{vandervaart2000asymptotic} and assumed convergence of the nuisance parameters;
    \item[second-order remainder:] $\hat{\bbeta}_n^* - \bbeta_0 + P_0 [D^*(P_n^*)] = o_P(n^{-1/2})$ by assumption.
\end{description}
This completes the proof.
\end{proof}

\subsection{Supplementary Example}
\label{section:supplementary-example}

Here we present an adaptation of the running example introduced in Section~\ref{sec:stat:view} where the outcome $Y$ is continuous. The form of the EIF for $B$ is the same as in the main example. The form of the fluctuation models for TMLE, however, requires more work.

Let $Z = (A, Y)$, where $A$ is a binary treatment indicator and $Y \in \reals$ is a \textit{continuous} outcome. Set arbitrarily $P \in \model$. Let $g_P(a, x) := P\left(A = a \middle| X = x\right)$ for any $(a,x) \in \{0,1\} \times \mathcal{X}$, and suppose that $g_P(1, X) > 0$ holds $P$-almost-surely. For both $a\in\{0,1\}$ let $\bar{Q}_P^{(a)}(X) = \E_P\left[ Y \middle| A = a, X\right]$, which is defined almost surely on $\mathcal{X}$. Define $\PsiP^{\flat}(X) = \bar{Q}^{(1)}_P(X) - \bar{Q}^{(0)}_P(X)$. Following the main example, let $B^\flat$ be given by
\begin{align}
    B\flat(P) = \argmin_{\bbeta \in \mathcal{B}} \E_P\left[ \left( \PsiP(X) - \bbeta^\top (1, V)^\top\right)^2 \right].
\end{align}

\subsubsection{Frequentist TMLE}
For any $P \in \model$ with corresponding $\PsiP^\flat$, $\bar{Q}_P$, $Q_{P}$ and $\eta_P$, characterize the fluctuation model with transformation $x \mapsto \phi(x) = x$ as follows: for all $\bepsilon \in \reals^p$,
\begin{align}
    \QbarFluctuation{P}{\bepsilon}^{(0)}(O) &= \bar{Q}_P^{(0)}(X) + \cleverH_0(O) \bepsilon^\top (1, V)^\top , \\
    \QbarFluctuation{P}{\bepsilon}^{(1)}(O) &= \bar{Q}_P^{(1)}(X) + \cleverH_1(O) \bepsilon^\top (1, V)^\top, \\
    \QFluctuation{P}{\bepsilon}(X) &= C(\bepsilon) \exp\left( \bepsilon^\top  (\Psi^{\flat}_P(X) - B^{\flat}(P)^\top (1, V)^\top) (1, V)^\top \right) Q_X(X).
\end{align}
The clever covariates are given by
\begin{align}
    \cleverH_0(O) &= -\frac{I(A = 0)}{g_P(0, X)}, \\
    \cleverH_1(O) &= \phantom{-}\frac{I(A = 1)}{g_P(1, X)},
\end{align}
and the loss functions by
\begin{align}
    \tmleLoss_0(t, O)   &= \frac{1}{2} I(A = 0)(Y - t)^2, \\
    \tmleLoss_1(t, O)   &= \frac{1}{2} I(A = 1)(Y - t)^2.
\end{align}
To ensure this is a valid setup, we need to check conditions (L1), (L2), and (M1).
Condition \ref{condition: well-defined-loss} is satisfied for this choice of loss function \citep{van2011targeted}. The derivatives of the loss functions are 
\begin{align}
    \dot{\tmleLoss}_0(t, O) &= -I(A = 0)(Y -t), \\
    \dot{\tmleLoss}_1(t, O) &= -I(A = 1)(Y - t).
\end{align}
Therefore, \ref{condition:loss-differentiable} is satisfied. To check \ref{condition:well-defined-fluctuation}, see that $(\phi^{-1})'(x) = 1$, so
\begin{align}
    &\dot{\tmleLoss}_0(\bar{Q}_P^{(0)}(X), O) (\phi^{-1})'(\phi(\bar{Q}_P^{(0)}(X))) \cleverH_0(O) +  \dot{\tmleLoss}_1(\bar{Q}_P^{(1)}(X), O)(\phi^{-1})'(\phi(\bar{Q}_P^{(1)}(X))) \cleverH_1(O) \\
    &= \dot{\tmleLoss}_0(\bar{Q}_P^{(0)}(X), O) \cleverH_0(O) +  \dot{\tmleLoss}_1(\bar{Q}_P^{(1)}(X), O) \cleverH_1(O) \\
    &= -\left\{ \frac{I(A = 1)}{g_P(1, X)} - \frac{I(A = 0)}{g_P(0, X)} \right\} (Y - \bar{Q}_P^{(A)}(X)) \\
    &= -\Delta_P^*(X).
\end{align}

\subsubsection{Bayesian TMLE}
The TMLE loss functions $\mathcal{L}_0$ and $\mathcal{L}_1$ in this example can be interpreted as the conditional log-likelihoods of $Y$ given either $A=0$ or $A = 1$ and $X$ under Gaussian distributions with fixed unit variance. This suggests constructing each $F_{\bepsilon}$ such that $Y$ given $A$ and $X$ follows a Gaussian distribution.
Fix arbitrarily $P \in \model$, $\bepsilon \in \reals^p$, introduce $\sigma_P^2(A, X) = \mathrm{Var}_P\left[Y \middle| A, X\right]$, and let $F_{\bepsilon}$ be characterized by
\begin{align}
    X &\sim Q_{P, \bepsilon}, \\
    A \mid X &\sim \mathrm{Bernoulli}(1/2), \\
    Y \mid X, A &\sim \mathrm{N}\left(\bar{Q}_{P, \bepsilon}^{(A)}(X), \sigma_P^2(A, X)\right)
\end{align}
where
\begin{align}
    Q_{P,\bepsilon}(X) &= C(\bepsilon)\exp\left( M^{-1} \bepsilon^\top (\Psi^{\flat}_P(X) - B^\flat(P)^\top (1, V)^\top) (1, V)^\top \right) Q_X(X), \\
    \bar{Q}_{P, \bepsilon}^{(0)}(X) &= \bar{Q}_P^{(0)}(X) + \sigma_P^2(0, X) H_0(O)  \bepsilon^\top M^{-1} (1, V)^\top, \\
    \bar{Q}_{P, \bepsilon}^{(1)}(X) &= \bar{Q}_P^{(1)}(X) + \sigma_P^2(1, X) H_1(O)  \bepsilon^\top M^{-1} (1, V)^\top. 
\end{align}
Then the conditional log-likelihood of $O$ under $F_{\bepsilon}$ is given by
\begin{align}
    \log f(O \mid \bepsilon) &= -\log(\sigma_P^2(A, X)) -\frac{1}{2 \sigma_P^2(A, X)} \left( Y -  \bar{Q}^{(A)}_{P,\bepsilon}(X) \right)^2 + \log Q_{P, \bepsilon}(X) + \text{constant} \\
    &= -\log(\sigma_P^2(A, X)) -\frac{1}{\sigma_P^2(A, X)} \left( \mathcal{L}_0(\bar{Q}_{P, \bepsilon}^{(0)}(X), O) + \mathcal{L}_1(\bar{Q}_{P, \bepsilon}^{(1)}(X), O) \right)\\
    & \qquad + \log Q_{P, \bepsilon}(X) + \text{constant}.
\end{align}
The gradient of the log-likelihood evaluated at $\bepsilon = \bzero$ is then
\begin{align}
    \frac{\partial}{\partial \bepsilon} \log f(O|\bepsilon) \evalat{\bepsilon = \bzero} =& \left\{ \frac{I(A = 1)}{g_P(1, X)} - \frac{I(A = 0)}{g_P(0, X)} \right\}(Y - \bar{Q}_P^{(A)}(X))M^{-1}(1, V)^\top \\
    &+ (\Psi_P^{\flat}(X) - B^\flat(P)^\top (1, V)^\top)M^{-1} (1, V)^\top,
\end{align}
which is equal to $D^*(P)(O)$. The Hessian of the log-likelihood evaluated at $\bepsilon = \bzero$ is given by
\begin{align}
    \frac{\partial^2}{\partial \bepsilon^2} \log f(O|\bepsilon) \evalat{\bepsilon = \bzero} =& \left( \frac{\sigma^2_P(A, X)}{g_P(A, X)^2}\right) M^{-1} (1, V)^\top(1, V) M^{-1} \\
    &+ (\Psi_P(X)^\flat - (B(P)^{\flat})^\top (1, V)^\top)^2 M^{-1} (1, V)^\top(1, V) M^{-1}.
\end{align}
Note that
\begin{align}
    P_0\left[\frac{\partial^2}{\partial \bepsilon^2} \log f(O|\bepsilon) \evalat{\bepsilon = \bzero}\right] = P_0[\lambda^*(P_0)].
\end{align}
Therefore conditions \ref{condition:bernstein-von-mises-gradient} and 
\ref{condition:bernstein-von-mises-hessian} of Theorem \ref{theorem:bvm:full} are satisfied for this fluctuation model.

\subsection{Statement and Proof of Theorem \ref{theorem:bvm:full}}
\label{section:bvm-proof}

\begin{definition}[Local Asymptotic Normality {\cite[Definition 7.14]{vandervaart2000asymptotic}}]
The sequence of statistical models $\{ P_{n,\theta} : \theta \in \Theta \}$ is locally asymptotically normal (LAN) at $\theta$ if there exist matrices $r_n$ and $I_{\theta}$ and random vectors $\Delta_{n,\theta}$ such that $\Delta_{n,\theta} \rightsquigarrow N(0, I_{\theta})$ and for every converging sequence $h_n \rightarrow h$
    \begin{align}
        \log \frac{dP_{n,\theta + r_n^{-1}h_n}}{dP_{n,\theta}} = h^\top \Delta_{n,\theta} - \frac{1}{2} h^\top I_{\theta} h + o_{P_{n,\theta}}(1).
    \end{align}
\end{definition}

It is sometimes more practical to show that $\{ P_{n,\theta}: \theta \in \Theta \}$ is differentiable in quadratic mean, which implies local asympotic normality \citep[Chapter 7]{vandervaart2000asymptotic}. 

Let
\begin{align}
    f_n^0\left(O_{1:n} \middle| \bepsilon \right) &:= \prod_{i=1}^n f\left(O_i \middle| \bepsilon, \bar{Q}_0, Q_0, \eta_0, \sigma_0 \right), \\
    \Pi_\epsilon^0\left(\bepsilon \middle| O_{1:n} \right) &:= \Pi_\epsilon\left(\bepsilon \middle| O_{1:n}, \bar{Q}_0, Q_0, \eta_0, \sigma_0 \right), \\
    \Pi_\beta^0\left(\bbeta \middle| O_{1:n} \right) &:= \Pi_\beta\left(\bbeta \middle| O_{1:n}, \bar{Q}_0, Q_0, \eta_0, \sigma_0 \right),
\end{align}
 be the likelihood, posterior for $\bepsilon$, and posterior for $\bbeta$ corresponding to the submodel $\mathcal{F}^0 = \{F_{\bepsilon}^{0} : \bepsilon \in \reals^{p}\}$ -- that is, the submodel through $P_0$ satisfying \eqref{eq-likelihood-condition}. Introduce $\vartheta_{0} : \bepsilon \mapsto B(F_{\bepsilon}^{0})$.

\begin{theorem}[Oracular Bernstein von-Mises]
\label{theorem:bvm:full}
Let us assume that:
    \begin{enumerate}
        \item \label{condition:bernstein-von-mises-lan} $\mathcal{F}^0$ 
        is locally asymptotically normal.
        \item\label{condition:vartheta} The mapping $\vartheta_{0}$ is invertible and its inverse $\vartheta_{0}^{-1}$ is twice differentiable at $\bbeta_{0} = B(P_0)$.
        \item \label{condition:bernstein-von-mises-tests} For every $\delta > 0$, there exists a sequence of tests $(\phi_n)_{n \geq 1}$ such that
        \begin{align}
            P_0 [\phi_n] \to 0 \text{ and } \sup_{\norm{\bbeta - \bbeta_0} \geq \delta} F_{\bepsilon}^0[(1 - \phi_n)] \to 0.
        \end{align}
        \item \label{condition:bernstein-von-mises-differentiability} The following holds $P_0$-almost surely: the map $\bepsilon \mapsto \log f^0_{n}(O | \bepsilon)$ is twice differentiable at every $\bepsilon \in \reals^p$ with gradient $\frac{\partial}{\partial \bepsilon} \log f_n^0(O \mid \bepsilon) \in \mathbb{R}^p$ and Hessian $\frac{\partial^2}{\partial \bepsilon^2} \log f_n^0(O | \bepsilon) \in \reals^{p \times p}$.
        \item \label{condition:bernstein-von-mises-gradient} The gradient satisfies
        \begin{align}
            \frac{\partial}{\partial \bepsilon} \log f_n^0(O | \bepsilon) \evalat{\bepsilon = \bzero} &= D^*(P_0)(O).
        \end{align}
        \item \label{condition:bernstein-von-mises-hessian} Define
        \begin{align}
            \Lambda_n^0(O) := \frac{\partial^2}{\partial \bepsilon^2} \log f_n^0(O | \bepsilon) \evalat{\bepsilon = \bzero}.
        \end{align} Then
        \begin{align}
            P_0[\Lambda_n^0] = P_0[\lambda^*(P_0)],
        \end{align}
        where $P_0[\lambda^*(P_0)]$ is assumed to be non-singular.
        \item \label{condition:bernstein-von-mises-prior} The joint prior density $\pi_\beta$ is absolutely continuous in a neighborhood of $\bbeta_0$ and $\pi_\beta(\bbeta_0) > 0$.
    \end{enumerate}
    Then
    \begin{align*}
        \norm{\Pi^0_{\sqrt{n}(\bbeta - \bbeta_0)} \left( \cdot \mid O_{1:n} \right) - N\left(\Delta_n^0, P_0[\lambda^*(P_0)]\right)}_1 = o_P(1)
    \end{align*}
    where
    \begin{align}
        \Delta_{n}^0 = \frac{1}{\sqrt{n}} \sum_{i=1}^n P_0[\lambda^*(P_0)]^{-1} D^*(P_0)(O_i).
    \end{align}
\end{theorem}

Note that condition \ref{condition:bernstein-von-mises-tests} can be difficult to verify. However, if $\bepsilon$ belongs to a compact set $E$, if $\mathcal{F}^{0}=\{ F_{\bepsilon}^0 : \bepsilon \in E \}$ is  identifiable (meaning $F_{\bepsilon}^0 = F_{\bepsilon'}^0$ implies $\bepsilon=\bepsilon'$) and if $\bepsilon \mapsto F_{\bepsilon}^{0}$ is continuous, then condition \ref{condition:bernstein-von-mises-tests} is not necessary \citep[Chapter 10]{vandervaart2000asymptotic}. 

\begin{proof}
First, we derive the Fisher information matrix of $\bbeta$. The log-density of $\bbeta$ is given by
\begin{align}
    \log f_n^0(O_{1:n} \mid \bbeta) &= \log f_n^0(O_{1:n} \mid \bepsilon = \vartheta_0^{-1}(\bbeta)).
\end{align}
Next, calculate
\begin{align}
    \frac{\partial^2}{\partial \bbeta^2} \log f_n^0&\left(O \middle| \bepsilon = \vartheta_0^{-1}(\bm{\beta})\right) \evalat{\bbeta = \bbeta_0} \\ 
    &= \left[\frac{\partial^2}{\partial \bbeta^2}  \vartheta_0^{-1}(\bbeta) \evalat{\bbeta = \bbeta_0} \right] \frac{\partial}{\partial \bepsilon} \log f_n^0\left(O \middle| \bepsilon\right) \evalat{\bepsilon = \bzero} \frac{\partial}{\partial \bepsilon} \log f_n^0\left(O \middle| \bepsilon\right) \evalat{\bepsilon = \bzero}^\top \\
    \label{eq:second-order-term}
    &\quad+ 
        \left[\frac{\partial}{\partial \bbeta} \vartheta_0^{-1}(\bbeta) \evalat{\bm{\beta} = \bbeta_0} \right] 
        \frac{\partial^2}{\partial \bepsilon^2} \log f_n^0\left(O_i \middle| \bepsilon\right) \evalat{\bepsilon = \bzero}
        \left[\frac{\partial}{\partial \bbeta} \vartheta_0^{-1}(\bbeta) \evalat{\bbeta = \bbeta_0} \right].
\end{align}
The first term (the product of a 3D tensor and a matrix) equals zero because $\bepsilon = \bzero$ maximizes $\bepsilon \mapsto \log f_n^0(O|\bepsilon)$, so
\begin{align}
    \frac{\partial}{\partial \bepsilon} \log f_n^0\left(O \middle| \bepsilon\right) \evalat{\bepsilon = \bzero} = \bzero.
\end{align}
The inner Hessian matrix of the second term equals $\Lambda_n^0(O)$. The surrounding matrices of the second term satisfy
\begin{align}
    \frac{\partial}{\partial \bbeta} \vartheta_0^{-1}(\bbeta) \evalat{\bbeta = \bbeta_0} = \left[ \frac{\partial}{\partial \bepsilon} \vartheta_0(\bepsilon) \evalat{\bepsilon = \bzero} \right]^{-1}.
\end{align}
Recall that $\vartheta_0(\bepsilon) = B(F_{\bepsilon}^0)$. Therefore, since $B$ is pathwise differentiable,
\begin{align}
    \frac{\partial}{\partial \bepsilon} \vartheta_0(\bepsilon) \evalat{\bepsilon = 0} &= P_0\left[ D^*(P_0)s^\top \right],
\end{align}
where $s \in (L_0^2(P_0))^p$ is the score of $F_n^0$ at $\bepsilon = \bzero$ that is, by assumption on the submodel $\mathcal{F}^0$, $s = D^*(P_0)$. Therefore
\begin{align}
    \frac{\partial}{\partial \bbeta} \vartheta_0^{-1}(\bbeta) \evalat{\bbeta = \bbeta_0} &= \left\{ P_0\left[D^*(P_0) D^*(P_0)^\top\right] \right\}^{-1} \\
    &= P_0[\lambda^*(P_0)]^{-1},
\end{align}
and
\begin{align}
    \frac{\partial^2}{\partial \bbeta^2} \log f_n^0&\left(O \middle| \bepsilon = \vartheta_0^{-1}(\bm{\beta})\right) \evalat{\bbeta = \bbeta_0} = P_0[\lambda^*(P_0)]^{-1} \Lambda_n^0(O) P_0[\lambda^*(P_0)]^{-1}.
\end{align}
The Fisher Information is then
\begin{align}
    I_{\bbeta_0} &:= \E_{P_0}\left[  \frac{\partial^2}{\partial \bbeta^2} \log f_n^0\left(O \middle| \bepsilon = \vartheta_0^{-1}(\bm{\beta})\right) \evalat{\bbeta = \bbeta_0} \right] \\
    &= \E_{P_0}\left[P_0[\lambda^*(P_0)]^{-1} \Lambda_n^0(O) P_0[\lambda^*(P_0)]^{-1} \right] \\
    &= P_0[\lambda^*(P_0)]^{-1} P_0[\Lambda_n^0(O)] P_0[\lambda^*(P_0)]^{-1} \\
    &= P_0[\lambda^*(P_0)]^{-1} P_0[\lambda^*(P_0)] P_0[\lambda^*(P_0)]^{-1} \\
    &= P_0[\lambda^*(P_0)]^{-1}.
\end{align}
The result follows from \citet[Theorem 10.1]{vandervaart2000asymptotic}.
\end{proof}

\begin{example}{Example (cont'd)} 
 Applying Theorem \ref{theorem:bvm:full} requires checking each of the assumptions in the context of the example.
\begin{itemize}
    \item Assumption  \ref{condition:bernstein-von-mises-lan}: the likelihood $f_n^0$ is that of an exponential family satisfying the conditions of  \citep[Example 7.7]{vandervaart2000asymptotic}, therefore $\{ f_n^0(\cdot | \bepsilon) : \bepsilon \in \reals^p \}$ is locally asymptotically normal.
    \item Assumption \ref{condition:vartheta}: it is impossible to directly check this condition; we propose a diagnostic in Section \ref{subsec:diagnostics} that an oracle knowing $P_0$, and hence $\mathcal{F}^0$, could rely on to check the invertibility of $\vartheta_0$.
    \item  Assumption \ref{condition:bernstein-von-mises-tests}: we restrict $F_{\bepsilon}^0$ to be defined for $\bepsilon \in E := [-C, C]$ for an arbitrary (large) constant $C$. Because $E$ is compact, assumption  \ref{condition:bernstein-von-mises-tests} is no longer necessary.
    \item Assumption \ref{condition:bernstein-von-mises-differentiability}: $\bepsilon \mapsto \log f_n^0(O \mid \bepsilon)$ is infinitely differentiable.
    \item Assumption \ref{condition:bernstein-von-mises-gradient}: it is easily checked that 
    \begin{align}
        \frac{\partial}{\partial \bepsilon} \log f_n^0(O \mid \bepsilon) \evalat{\bepsilon = \bzero} &= D^*_{\CATEsymb}(P_0)(O).
    \end{align}
    \item Assumption \ref{condition:bernstein-von-mises-hessian}: the Hessian of $\bepsilon \mapsto \log f_n^0(O \mid \bepsilon)$ evaluated at $\bepsilon = \bzero$ is 
    \begin{align}
        \Lambda_n^*(O) =& \frac{1}{g_0(A, X)^2} \left(\bar{Q}_0^{(A)}(X) - 1\right)\bar{Q}_0^{(A)}(X)  M^{-1} (1, V) (1,V)^\top M^{-1} \\
        &+ (\psi_0^\CATEsymb(X) - B^\CATEsymb(P_0) (1, V)^\top)^2  M^{-1} (1, V) (1,V)^\top M^{-1},
    \end{align}
    and its expectation is given by
    \begin{align}
        P_0\left[\Lambda_n^*(O)\right] &= M^{-1} P_0 \left[ \left( \frac{\mathrm{Var}_P(Y \mid A, X)}{g_0(A, X)^2} + \left(\psi_0^\CATEsymb - B^\CATEsymb(P_0)(1, V)^\top \right)^2 \right) (1,V)(1,V)^\top \right] M^{-1},
    \end{align}
    which is the variance of the EIF $D^*(P_0)$.
    \item Assumption \ref{condition:bernstein-von-mises-prior}: this condition can be easily satisfied by proper choice of the prior density $\pi_{\beta}$.
 \end{itemize}
\end{example}

\subsection{Metropolis-Hastings Algorithm}
\label{section:metropolis-hastings}

Let 
\begin{equation*}
    f^*_n(O_{1:n} \mid \bepsilon) = \prod_{i=1}^n f(O_i | \bepsilon, \bar{Q}_n^*, Q_n^*, \eta_n, \sigma_n)
\end{equation*}
be the targeted likelihood  akin to $f_{n}^{0}(O_{1:n} \mid \bepsilon)$ that we defined and used  in Section~\ref{section:bvm-proof} (we simply substitute the estimated features for the true features). Let $\pi_{\epsilon}$ be a prior distribution for $\bepsilon$ as in~\eqref{eq-epsilon-prior}. Let $\bepsilon^{(0)}$
 be a starting value for the Markov Chain specified by the user, $\tau$ a tuning parameter, and $T$ the number of iterations. The algorithm is implemented on the log-scale to avoid numerical precision issues. For $t = 1, \dots, T$:
 \begin{enumerate}
     \item Draw $\bepsilon^\prime \sim \mathrm{N}(\bepsilon^{(t - 1)}, \tau \bm{I})$.
     \item Draw $u \sim \mathrm{Uniform}(0, 1)$.
     \item Let \begin{align}
         \ell^{t-1} &= \log f_n^*(O_{1:n} \mid \bepsilon^{(t - 1)}) + \log \pi_\epsilon(\bepsilon^{(t - 1)}), \\
         \ell^\prime &= \log f_n^*(O_{1:n} \mid \bepsilon^\prime) + \log \pi_\epsilon(\bepsilon^\prime).
     \end{align}
     \item Let $A^{(t)} = \mathds{I}(\ell^\prime - \ell^{t-1} > \log(u))$. Then
     \begin{align}
        \bepsilon^{(t)} &= \begin{cases}
            \bepsilon^{\prime} & \text{if } A^{(t)} = 1, \\
            \bepsilon^{(t-1)} & \text{otherwise.}
        \end{cases}
    \end{align}
    \item Set $\bbeta^{(t)} = B(\bepsilon^{(t)})$.
\end{enumerate}
The mean $\bar{A} = \frac{1}{T}\sum_{t=1}^T A^{(t)}$ is called the \textit{acceptance ratio}.

The value of $\tau$, the standard deviation of the proposal distribution, is found by a binary search procedure. Let $\tau_{\mathrm{min}}$ and $\tau_{\mathrm{max}}$ define the upper and lower search bounds for $\tau$. Let $K$ be the maximum number of iterations in the search. Let $\bar{A}(\tau)$ be the acceptance ratio from the Metropolis-Hastings algorithm given above with proposal distribution $\tau$ run for $1000$ iterations. Then iterate:
\begin{enumerate}
    \item Let $\tau_l = \tau_{\mathrm{min}}$ and $\tau_u = \tau_{\mathrm{max}}$. 
    \item For $k = 1, \dots, K$:
    \begin{enumerate}
        \item Let $\tau_k = (\tau_l + \tau_u) / 2$.
        \item Let $\bar{A}_k = \bar{A}(\tau_k)$.
        \item If $\bar{A}_k \in (0.3, 0.4)$, then break and return $\tau_k$.
        \item If $\bar{A}_k \leq 0.3$, set $\tau_u = \tau_k$.
        \item If $\bar{A}_k \geq 0.4$, set $\tau_l = \tau_k$.
    \end{enumerate}
    \item Return $\tau_K$.
\end{enumerate}
This procedure targets an acceptance ratio in between $30\%$ and $40\%$.

\subsection{Simulation Study Results}

\begin{tiny}
\begin{table}[ht]
    \centering
    \begin{tabular}{|r|l|r|r|r|r|}
    \hline
    \multicolumn{2}{|c}{} & \multicolumn{2}{c}{$\beta_1$} & \multicolumn{2}{c|}{$\beta_2$} \\
    \hline
    $N$ & Estimator & 95\% Coverage & Absolute Bias & 95\% Coverage & Absolute Bias \\
    \hline
    \multicolumn{6}{|l|}{(a): $g_P$ correctly specified, $\bar{Q}_P$ correctly specified} \\
    \hline
    50 & Frequentist & 0.87 & 0.099 & 0.78 & 0.110\\
    \hline
     & Bayesian & 0.90 & 0.100 & 0.85 & 0.110\\
    \hline
    100 & Frequentist & 0.92 & 0.069 & 0.91 & 0.066\\
    \hline
     & Bayesian & 0.94 & 0.069 & 0.95 & 0.067\\
    \hline
    250 & Frequentist & 0.93 & 0.043 & 0.94 & 0.041\\
    \hline
     & Bayesian & 0.94 & 0.043 & 0.95 & 0.042\\
    \hline
    500 & Frequentist & 0.96 & 0.030 & 0.95 & 0.030\\
    \hline
     & Bayesian & 0.97 & 0.030 & 0.97 & 0.030\\
    \hline
    750 & Frequentist & 0.95 & 0.025 & 0.96 & 0.022\\
    \hline
     & Bayesian & 0.96 & 0.024 & 0.97 & 0.023\\
    \hline
    1000 & Frequentist & 0.92 & 0.022 & 0.95 & 0.020\\
    \hline
     & Bayesian & 0.92 & 0.022 & 0.95 & 0.020\\
    \hline
    \multicolumn{6}{|l|}{(b): $g_P$ correctly specified, $\bar{Q}_P$ incorrectly specified} \\
    \hline
    50 & Frequentist & 0.93 & 0.100 & 0.88 & 0.110\\
    \hline
     & Bayesian & 0.95 & 0.100 & 0.95 & 0.110\\
    \hline
    100 & Frequentist & 0.96 & 0.068 & 0.90 & 0.080\\
    \hline
     & Bayesian & 0.98 & 0.068 & 0.96 & 0.081\\
    \hline
    250 & Frequentist & 0.96 & 0.042 & 0.92 & 0.048\\
    \hline
     & Bayesian & 0.98 & 0.042 & 0.95 & 0.049\\
    \hline
    500 & Frequentist & 0.95 & 0.032 & 0.94 & 0.035\\
    \hline
     & Bayesian & 0.97 & 0.032 & 0.95 & 0.035\\
    \hline
    750 & Frequentist & 0.96 & 0.025 & 0.94 & 0.028\\
    \hline
     & Bayesian & 0.97 & 0.025 & 0.95 & 0.028\\
    \hline
    1000 & Frequentist & 0.96 & 0.021 & 0.94 & 0.024\\
    \hline
     & Bayesian & 0.97 & 0.021 & 0.95 & 0.025\\
    \hline
    \multicolumn{6}{|l|}{(c): $g_P$ incorrectly specified, $\bar{Q}_P$ correctly specified} \\
    \hline
    50 & Frequentist & 0.88 & 0.100 & 0.86 & 0.099\\
    \hline
     & Bayesian & 0.89 & 0.100 & 0.88 & 0.100\\
    \hline
    100 & Frequentist & 0.92 & 0.068 & 0.92 & 0.067\\
    \hline
     & Bayesian & 0.91 & 0.068 & 0.93 & 0.068\\
    \hline
    250 & Frequentist & 0.94 & 0.040 & 0.92 & 0.042\\
    \hline
     & Bayesian & 0.95 & 0.040 & 0.94 & 0.042\\
    \hline
    500 & Frequentist & 0.95 & 0.029 & 0.96 & 0.029\\
    \hline
     & Bayesian & 0.95 & 0.029 & 0.96 & 0.029\\
    \hline
    750 & Frequentist & 0.96 & 0.024 & 0.95 & 0.024\\
    \hline
     & Bayesian & 0.96 & 0.024 & 0.94 & 0.024\\
    \hline
    1000 & Frequentist & 0.95 & 0.020 & 0.96 & 0.019\\
    \hline
     & Bayesian & 0.95 & 0.020 & 0.96 & 0.019\\
    \hline
    \multicolumn{6}{|l|}{(c): $g_P$ incorrectly specified, $\bar{Q}_P$ incorrectly specified} \\
    \hline
    50 & Frequentist & 0.89 & 0.120 & 0.85 & 0.120\\
    \hline
     & Bayesian & 0.90 & 0.120 & 0.89 & 0.120\\
    \hline
    100 & Frequentist & 0.88 & 0.087 & 0.92 & 0.076\\
    \hline
     & Bayesian & 0.88 & 0.087 & 0.94 & 0.077\\
    \hline
    250 & Frequentist & 0.90 & 0.057 & 0.94 & 0.047\\
    \hline
     & Bayesian & 0.89 & 0.057 & 0.96 & 0.047\\
    \hline
    500 & Frequentist & 0.82 & 0.049 & 0.95 & 0.032\\
    \hline
     & Bayesian & 0.81 & 0.049 & 0.96 & 0.032\\
    \hline
    750 & Frequentist & 0.77 & 0.046 & 0.96 & 0.028\\
    \hline
     & Bayesian & 0.76 & 0.046 & 0.97 & 0.028\\
    \hline
    1000 & Frequentist & 0.74 & 0.042 & 0.95 & 0.023\\
    \hline
     & Bayesian & 0.71 & 0.042 & 0.95 & 0.023\\
    \hline
    \end{tabular}
    \caption{Empirical coverage of 95\% credible (confidence) intervals and absolute bias of estimators in the simulation study.}
    \label{tab:simulation_results}
\end{table}
\end{tiny}
\end{document}